\documentclass[11pt,twoside,reqno,centertags,draft]{amsart}
\usepackage{amsmath,amsthm,amsfonts,amssymb,amsrefs}
\usepackage{enumerate}
\numberwithin{equation}{section}
\allowdisplaybreaks

\theoremstyle{definition} 
\newtheorem{definition}{Definition}[]
\newtheorem{rem}{Remark}[section]
\theoremstyle{plain} 
\newtheorem{thm}{Theorem}[section] 
\newtheorem{lem}[thm]{Lemma}
\newtheorem{cor}[thm]{Corollary}
\newtheorem{prop}[thm]{Proposition}

\newcommand{\Z}{\mathbb {Z}}

\newcommand{\R}{\mathbb{R}}

\newcommand{\supp}{\rm{supp}}

\def\div{ \hbox{\rm div}\,  }

\def\supp{\, \hbox{\rm supp}\,  }
\def\id{\hbox{\rm Id}}
\def\adj{\hbox{\rm adj}}
\def\tr{\hbox{\rm tr}}
\def\divA{\, \hbox{\rm div}_A\,  }
\def\divx{\, \hbox{\rm div}_x\,  }
\def\divy{\, \hbox{\rm div}_y\,  }

\def\da{\delta\!a}
\def\dK{\delta\!K}
\def\du{\delta\!u}
\def\dv{\delta\!v}
\def\dpsi{\delta\!\psi}

\newcommand{\Sum}{\displaystyle \sum}

\newcommand{\Frac}{\displaystyle \frac}

\def\cM{{\mathcal M}}
\def\cC{{\mathcal C}}
\def\cS{{\mathcal S}}
\let\tilde=\widetilde

\begin{document}

\title[Well-posedness of full compressible Navier-Stokes system]{On the well-posedness of the  full compressible Navier-Stokes system in critical Besov spaces}
\date{}

\author[N. Chikami]{Noboru Chikami}
\address[N. Chikami]
{Mathematical Institute
    Tohoku University, Sendai 980-8578, Japan.}
\email{sb0m23@math.tohoku.ac.jp}

\author[R. Danchin]{Rapha\"el Danchin}
\address[R. Danchin]
{Universit\'e Paris-Est, LAMA, UMR 8050 and Institut Universitaire de France,
 61 avenue du G\'en\'eral de Gaulle,
94010 Cr\'eteil Cedex, France.}
\email{danchin@univ-paris12.fr}

\begin{abstract} We are concerned with the Cauchy problem of the full compressible Navier-Stokes
equations satisfied by viscous and heat conducting fluids in $\R^n.$ We focus on the so-called critical 
Besov regularity framework. In this setting, it is natural to consider initial densities $\rho_0,$ velocity fields $u_0$ 
and temperatures $\theta_0$ with $a_0:=\rho_0-1\in\dot B^{\frac np}_{p,1},$ $u_0\in\dot B^{\frac np-1}_{p,1}$
and $\theta_0\in\dot B^{\frac np-2}_{p,1}.$ 
After recasting the whole system in Lagrangian coordinates, and working   with the \emph{total  energy along the flow}  
rather than with the temperature, we discover that the system may be solved by means
of Banach fixed point theorem in a critical functional framework whenever the space dimension
is $n\geq2,$ and $1<p<2n.$ Back to Eulerian coordinates, this allows to improve the range of $p$'s for which the system is locally well-posed, compared to \cite{danchin1}. 
\end{abstract}

\maketitle

\section{Introduction}
We consider the Cauchy problem of 
the following full compressible 
Navier-Stokes equations 
in $\mathbb{R}^n$, $n\geq 2$:
\begin{equation}\label{CNS}\left\{
\begin{array}{lr}
\partial_t \rho + \div(\rho u) = 0, 
   &\quad (t,x) \in \R_+ \times \R^n, \\[1ex]
\partial_t (\rho u) + \div(\rho u\otimes u) + \nabla P
    =  \div\tau, 
   &\quad (t,x) \in \R_+ \times \R^n, \\[1ex]
 \partial_t \Big[\rho\Big(\Frac{|u|^2}{2}+e\Big)\Big] 
  + \div\Big[u\Big(\rho\big(\Frac{|u|^2}{2}+e\big)+P\Big)\Big] &\\
\qquad\qquad\qquad\qquad\qquad\qquad\qquad
  = \div(\tau\cdot u) - \div q, &\quad (t,x) \in \R_+ \times \R^n,\\[1ex]
(\rho,u, \theta)|_{t=0} = (\rho_0,u_0,\theta_0), &\quad x\in \R^n,
\end{array}
\right.
\end{equation} 
where $\rho = \rho(t,x)\in\R_+$, $u=u(t,x)\in\R^n$ and $e=e(t,x)\in\R$ 
are the unknown functions, representing the fluid density, 
the velocity vector field and the internal energy per unit mass, respectively. 
We restrict ourselves to the case of Newtonian gases, namely we assume the viscous stress tensor $\tau$ to be given by
\begin{equation*}
\tau := \lambda \div u \ \id + 2\mu D(u), 
\end{equation*}
where $D(u)$ designates the deformation tensor 
defined by 
\begin{equation*}
D(u):=\frac{1}{2}(Du+\nabla u) \quad \text{with} \quad
(Du)_{ij}:=\partial_j u^i \quad \text{and} \quad
(\nabla u)_{ij} := (^t\!(Du))_{ij} = \partial_i u^j.
\end{equation*}
The viscosity coefficients $\lambda=\lambda(\rho)$ and $\mu=\mu(\rho)$ are given smooth functions of $\rho$ satisfying 
$\mu>0$ and $\lambda + 2 \mu>0$, which ensures the ellipticity of the second order
operator in the velocity equation. 
We assume the Fourier law; that is the heat conduction $q$ is given by $q=-k\nabla \theta$ where 
$k=k(\rho)$ is a given positive smooth function and $\theta=\theta(t,x)$, the temperature.
We also suppose that the gas obeys 
Joule's law, namely that $e$ is a
function of $\theta$ only; for simplicity we 
assume $e=C_v \theta$ for a (positive) 
specific heat constant $C_v$. 
The given function $P$ represents 
the pressure depending on $\rho$ and $\theta$. In that paper, we restrict ourselves to the following pressure law: 
\begin{equation*}
P(\rho,\theta):=\pi_0(\rho)+\theta\pi_1(\rho),
\end{equation*}
where $\pi_0$ and $\pi_1$ are given smooth functions. 

Important examples of such pressure laws are ideal fluids (for which $\pi_0(\rho)=0$ 
and $\pi_1(\rho)=R\rho$ for some positive constant $R$), 
 barotropic gases ($\pi_1(\rho)=0$) and  Van der Waal gases 
($\pi_0(\rho)=-\alpha\rho^2$ 
and $\pi_1(\rho)= \beta\rho/(\gamma-\rho)$ 
for some positive constant $\alpha$, $\beta,$ $\gamma$).
\medbreak
The boundary conditions at infinity are that $u$ and $\theta$ tend to $0,$
and that $\rho$ tends to some positive constant $\rho^*.$ The exact meaning of the convergence
will follow from the functional  framework we shall work in. 
For simplicity, we assume $C_v=1$ and 
$\rho^*=1$ in all that  follows. With no loss of generality, one can impose in addition that  $\pi_0(1)=0.$


\subsection{Aim of the paper}

Our main goal is to solve the full Navier-Stokes equations 
in the  so-called \emph{critical regularity framework}. 
This approach originates from a paper  of Fujita-Kato \cite{fujita-kato} devoted 
to the well-posedness issue for the incompressible Navier-Stokes equations. 
In our context, the idea is to solve \eqref{CNS} in a functional space having 
the same invariance by time and space dilations as \eqref{CNS}, namely
$(\rho,u,\theta)\to (\rho_{\nu},u_{\nu},\theta_{\nu})$ with
\begin{equation}
\label{eq:critical}
\rho_{\nu} (t,x)=\rho(\nu^2 t,\nu x), \quad
u_{\nu} (t,x)=\nu u(\nu^2 t,\nu x) 
\ \text{ and }\ 
\theta_{\nu} (t,x)=\nu^2 \theta(\nu^2 t,\nu x).
\end{equation}
The above family of transforms does not quite leave \eqref{CNS} invariant
(as $P$ has to be changed into $\nu^2 P$). Nevertheless, the pressure term is, to some extent,  lower order,
and it is thus suitable to address the  solvability issue of  the system in `critical'  spaces, that is in spaces
with  norm invariant for all $\nu>0$  by the scaling transformation 
$(\overline\rho,\overline u,\overline K)
\to (\overline\rho_{\nu},\overline u_{\nu},\overline K_{\nu}).$
\medbreak

Following recent works dedicated to this issue (see e.g. \cites{danchin00,danchin1}),
we here employ homogeneous Besov spaces \emph{with summation index $1$}.
The main reasons why are  that those spaces have nice embedding properties that fail to be true in e.g. Sobolev spaces, 
and are particularly well adapted to the study of systems related to the heat equation (which is the case here for the
velocity and energy equations) as they allow to gain two full derivatives with respect to the data, after taking a $L^1$ norm
in time (see Section \ref{s:parabolic} below). 
\medbreak
Before giving more insight on our main result, let us  recall the definition of Besov spaces with last index $1.$ 
Hereafter, we denote by $L^p\  (1\le p\le\infty)$ standard Lebesgue spaces on 
$\R^n$, and by $\ell^p$ the set of sequences with summable $p$-th powers.  
Let $\{ \phi_j \}_{j\in \Z}$ be a Littlewood-Paley 
dyadic decomposition.
Namely, let $\phi \in \mathcal{S}$ be a non-negative radially 
symmetric function that satisfies
$$\displaylines{\supp \widehat{\phi} \subset \{ \xi \in \R^n; 2^{-1} < |\xi |< 2\}, \cr
  \widehat{\phi_j} 
    :=\widehat{\phi} (2^{-j}\xi) \ (^\forall j\in \Z)
\quad \text{and} \quad 
     \sum_{j\in \Z} \widehat{\phi_j}(\xi) = 1 
 \ (^\forall \xi\neq 0).}
 $$
We further set 
$\displaystyle \widehat{\Phi}(\xi):= 1 - \sum_{j \ge 1} \widehat{\phi_j} (\xi)$
and $\dot S_mu:=\Phi(2^{-m}\cdot)\ast u,$ for $m\in\Z.$

\begin{definition}[Homogeneous Besov spaces]
Let $\mathcal{S}'(\R^n)$ be the space of tempered distributions on $\R^n.$ 
For  $1\le p\le \infty$ and $s\leq n/p,$ we denote by $\dot B^s_{p,1}(\R^n)$ 
(or more simply $\dot B^s_{p,1}$) the space of tempered distributions $u$ 
so that \footnote{See e.g \cite{bahouri-chemin-danchin} 
or \cite{triebel} for more details on the Besov spaces.}
$$
u=\sum_{j\in\Z}\dot\Delta_ju\quad\hbox{in}\quad\cS'(\R^n)\quad\hbox{with }\ \dot\Delta_ju:=  \phi_j \ast u
$$
and 
$$
\|u \|_{\dot{B}^{s}_{p,1}} :=\sum_{j\in\Z} 2^{js} \| \dot\Delta_ju\|_{L^p}<\infty.
$$
\end{definition}
\noindent
In this framework, it is clear that data $\rho_0=1+a_0,$ $u_0$ and $\theta_0$  corresponding to the scaling invariance
\eqref{eq:critical} have to be taken as follows:
$$
a_0 \in \dot B^{\frac{n}{p}}_{p,1},\quad u_0\in\dot B^{\frac{n}{p}-1}_{p,1}
\ \text{ and } \ \theta_0\in\dot B^{\frac{n}{p}-2}_{p,1}.
$$
Let us recall that in the barotropic case,  the critical Besov regularity was first considered by 
the latter author in a $L^2$ type  framework to obtain a global solution \cite{danchin00} for small perturbations
of a stable constant state $(\rho^*,0)$ with $\rho^*>0.$ 
Since then, there have been a number of refinements as regards admissible exponents for the global existence 
(see \cites{charve-danchin,chen-miao-zhang10} and the 
references therein). 
The local-in-time existence issue in the critical regularity framework with both large $u_0$ and $a_0$ (with 
$\rho_0$ bounded away from $0$)   has been addressed only in the barotropic case. 
The proof either involves the time-weighted norm  or 
the frequency localization techniques (see \cites{chen-miao-zhang1,danchin07} 
and \cite{haspot} for their generalization). The slightly nonhomogeneous  case (density  close to some constant) 
is easier and has been investigated for the full Navier-Stokes equations as well in \cite{danchin1}. 
\medbreak
When solving  \eqref{CNS} or its barotropic version bluntly, the main difficulty   is 
that the system is only partially parabolic, owing to the mass conservation equation which 
is of hyperbolic type. This  precludes any attempt to use the Banach fixed point theorem in a suitable space.
 As a matter of fact,  existence may be proved either through compactness
methods, or through a high norm uniform bounds / low norm stability estimates  scheme, as in the case of 
quasilinear symmetric hyperbolic systems. Another drawback of this direct approach is that the  loss of regularity in the stability estimates considerably restricts the set of data for which uniqueness may be proved (see  Chap. 10 of \cite{bahouri-chemin-danchin} for more details).  
\medbreak
Prompted by the recent paper 
dedicated to the compressible barotropic flow \cites{danchin2}
or by the work in \cites{danchin-mucha} concerning incompressible inhomogeneous fluids,
 we here  aim at solving the full compressible system 
\eqref{CNS} in the {\sl Lagrangian} coordinates. 
Let us emphasize that this  approach has already been successfully applied 
in the case of smooth data (see e.g. \cites{mucha, nash, valli, valli-zajaczkowski}).
We here want to perform it \emph{in the critical regularity framework}.

The motivation behind introducing  Lagrangian coordinates 
is  to effectively  eliminate the hyperbolic part of the system, 
given that the density equation becomes explicitly solvable 
once the flow of the velocity field has been  determined. 
At the same time, the system  for the velocity and energy in Lagrangian coordinates
remains of parabolic type (at least for small enough time), 
and  the Banach fixed point theorem turns out to be  applicable  for obtaining 
the existence \emph{and} uniqueness of the solution in {\sl the same class of spaces} as in the Eulerian framework. 
This is the key to improving the set of data leading to well-posedness, compared to  \cite{danchin1}.


\subsection{Notation}
Before introducing the Lagrangian system, 
let us  list some notational conventions. 
Throughout the paper, we denote by $C$ a generic harmless `constant' the value of which may vary from line to line. 
The notation $A\lesssim B$ means that $A\le C B$. 
For a $\cC^1$ function $F:\R^n\rightarrow\R^n\times\R^m$, 
we define $\div F : \R^n\rightarrow\R^m$ by 
\begin{equation*}
(\div F)^j:=\sum_{i} \partial_i F_{ij}, \quad 1\le j\le m.
\end{equation*}
For $n\times n$ matrices $A=(A_{ij})_{1\le i,j\le n}$ 
and $B=(B_{ij})_{1\le i,j\le n}$, 
we define the trace product $A:B$ by
\begin{equation*}
A:B = \tr AB = \sum_{ij} A_{ij} B_{ji}.
\end{equation*}
We denote by $\adj(A)$  the adjugate matrix of $A$,  i.e. the transpose of the cofactor matrix of $A$. 
Given some matrix $A$, we define the 
``twisted" deformation tensor and divergence operator 
(acting on vector fields $z$) by the formulae 
\begin{equation*}
D_A(z):=\frac{1}{2}(Dz \cdot A + {}^t\!A\cdot\nabla z),
\end{equation*}
\begin{equation*}
\divA z:= ^t\!A:\nabla z = Dz:A.
\end{equation*}
The flow $X_u$ of the time dependent vector field $u$ is (formally) defined as the solution to 
\begin{equation}
\label{def-flow}
X_u(t,y)=y+\int_0^t u(\tau,X_u(\tau,y)) \,d\tau.
\end{equation}

 We denote by $E$ the {\sl total energy by unit volume} of the fluid, that is, remembering that 
 $e=C_v\theta$ and that $C_v=1,$
 \begin{equation}
\label{def-E}
E := \rho\Bigl(\frac{|u|^2}{2}+e\Bigr)= \rho\Bigl(\frac{|u|^2}{2}+\theta\bigr)\cdotp 
\end{equation}
With the new set of unknowns $(\rho,u,E)$, 
the system \eqref{CNS} is converted to
\begin{equation}
\label{CNSE}
\left\{
\begin{array}{l}
\partial_t \rho + \div(\rho u) = 0,\\[1ex]
\partial_t(\rho u) - \div\tau + \div (\rho u \otimes u) + \nabla P
    = 0 , \\[1ex]
\partial_t E + \div (uE) 
   - \div\Big[k(\rho) \nabla \Big(\Frac{E}{\rho} \Big)\\\qquad\qquad
   +\tau\cdot u -k(\rho) \nabla \Big(\Frac{|u|^2}{2}\Big)
   - u \pi_0(\rho) -u\Big(\Frac{E}{\rho} - \Frac{|u|^2}{2}\Big) \pi_1(\rho) \Big] =0.
\end{array}
\right.
\end{equation}


\subsection{Lagrangian coordinates}
Let  $\bar\rho(t,y):=\rho(t,X_u(t,y))$, 
$\bar u(t,y):=u(t,X_u(t,y))$
and $\bar E(t,y):=E(t,X_u(t,y))$ denote the density, velocity and energy functions in Lagrangian coordinates.
Setting $J=J_u:=\det(DX_u)$ and $A=A_u:=(DX_u)^{-1}$,  it is shown in Appendix that System \eqref{CNSE} recasts in 
\begin{equation}\label{CNSE-l}
\left\{\begin{array}{l}
\partial_t (J\overline\rho) = 0, \\[2ex]
\rho_0 \partial_t \overline u 
 -\div\Big(    \adj(DX)(2\mu(\overline\rho) D_A \overline u 
     + \lambda(\overline\rho)\divA \overline u 
     - P(\overline\rho,\overline E)\id\big) \Big)= 0, \\[2ex]
\partial_t(J\overline E)
 -\div\!\Big(\adj(DX) \big(k(\overline\rho) {}^t\!A \ \nabla (\frac{\overline E}{\overline\rho})
 +\overline\tau \cdot \overline u
 -k(\overline\rho) {}^t\!A \ \nabla (\frac{|\overline u|^2}{2})-\overline u\overline P(\overline\rho,\overline E)=0, \\[2ex]
(\overline\rho,\overline u,\overline E)|_{t=0} = (\rho_0,u_0,E_0).
\end{array}
\right.
\end{equation} 
Looking at the energy equation, it is thus natural to introduce the
 {\sl total energy along the flow} defined by
\begin{equation}
\label{def-K}
\overline K:=J\overline E=\rho_0(\bar \theta+\frac{|\bar u|^2}{2}).
\end{equation}
We shall thus eventually consider the following system 
\begin{equation}
\label{CNS-l}
\left\{
\begin{array}{l}
\partial_t (J\overline\rho) = 0, \\[2ex]
\rho_0 \partial_t \overline u 
 -\div\Big[
    \adj(DX)(2\mu(\overline\rho) D_A \bar u 
     + \lambda(\overline\rho)\divA \bar u - \overline P(\overline\rho,\overline K)\id)
    \Big]= 0, \\[2ex]
\partial_t\overline K
 -\div\Big[\adj(DX) \big(k(\bar\rho) ^t\!A \nabla(\frac{\overline K}{\rho_0}) 
 -k(\bar\rho)^t\!A \nabla (\frac{|\bar u|^2}{2}) 
 +\bar\tau \cdot \bar u 
 -\bar u \overline P(\overline\rho,\overline K)
 \big) \Big]=0, \\[2ex]
(\overline\rho,\overline u,\overline K)|_{t=0} = (\rho_0,u_0,K_0),
\end{array}
\right.
\end{equation} 
where we have redefined the initial data $K_0$ as
\begin{equation}
\label{def-K0}
K_0 := E_0 = \rho_0 \Big(\theta_0+\frac{|u_0|^2}{2}\Big),
\end{equation}
and the pressure function $\overline P$ as 
\begin{equation}\nonumber
\overline P(\overline\rho,\overline K) 
:= \pi_0(\bar\rho)
 + \Bigl(\frac{\overline K}{\rho_0} - \frac{|\bar u|^2}{2}\Bigr) \pi_1(\bar\rho).
\end{equation}

Let us finally emphasize that one may forget any reference 
to the initial Eulerian vector-field $u$ by 
defining directly the ``flow" X of $\overline u$ by 
the formula 
\begin{equation}
\label{def-flow2}
X(t,y)=y+\int_0^t \overline u(\tau,y) \,d\tau.
\end{equation}


\subsection{Main results}
We shall obtain the existence and uniquenesss of a 
local-in-time solution $(\overline \rho, \overline u, \overline K)$
for \eqref{CNS-l}, with $\overline a:= \overline\rho-1$ 
in $\cC([0,T]:\dot B^{\frac{n}{p}}_{p,1})$ 
and $(\overline u, \overline K)$ in the space 
\begin{equation}
\label{def:space}
E_p(T):=\left\{ (v, \psi) \ \Bigg| \
\begin{array}{l}
 v\in \cC([0,T];\dot B^{\frac{n}{p}-1}_{p,1}), \ 
        \partial_t v, \nabla^2 v \in L^1(0,T ; \dot B^{\frac{n}{p}-1}_{p,1})\\[1ex] 
 \psi \in \cC([0,T];\dot B^{\frac{n}{p}-2}_{p,1}),
        \partial_t \psi, \nabla^2 \psi \in L^1(0,T ; \dot B^{\frac{n}{p}-2}_{p,1})           
\end{array}\right\}
\end{equation}
endowed with the norm
$$\|(v,\psi)\|_{E_p(T)}:=\|v\|_{L^\infty_T(\dot B^{\frac{n}{p}-1}_{p,1})} 
         + \|\partial_t v, \nabla^2 v\|_{L^1_T(\dot B^{\frac{n}{p}-1}_{p,1})} 
+\|\psi\|_{L^\infty_T(\dot B^{\frac{n}{p}-2}_{p,1})} 
         + \|\partial_t \psi, \nabla^2 \psi\|_{L^1_T(\dot B^{\frac{n}{p}-2}_{p,1})}. $$
It is easily checked that $E_p(T)$ is 
critical in the meaning of \eqref{eq:critical}. 
\medbreak
Let us now state our main result.
\begin{thm}\label{thm1}
Let $1<p<2n$ and $n\ge2$. Let $u_0$ be a vector field 
in $\dot B^{\frac{n}{p}-1}_{p,1}$ and 
$K_0,$ a real valued function in  $\dot B^{\frac{n}{p}-2}_{p,1}$. 
Assume that  $\rho_0$ satisfies 
$a_0:=(\rho_0-1)\in\dot B^{\frac{n}{p}}_{p,1}$ and
\begin{equation}
\label{cond-vacuum}
\inf_x \rho_0(x) > 0.
\end{equation}
Then System \eqref{CNS-l} admits a unique local 
solution $(\overline \rho, \overline u, \overline K)$ 
with $\overline\rho$ bounded away from zero, $\overline a:=\overline\rho-1$ 
in $\cC([0,T];\dot B^{\frac{n}{p}}_{p,1})$ 
and $(\overline u, \overline K)$ in $E_p(T)$.

Moreover, the flow map 
$(a_0,u_0,K_0) \mapsto (\overline a, \overline u, \overline K)$ 
is Lipschitz continuous from 
$\dot B^{\frac{n}{p}}_{p,1}
  \times \dot B^{\frac{n}{p}-1}_{p,1}
  \times \dot B^{\frac{n}{p}-2}_{p,1}$
to $\cC([0,T];\dot B^{\frac{n}{p}}_{p,1}) \times E_p(T)$.
\end{thm}

In Eulerian coordinates, the above theorem implies:
\begin{thm}\label{thm2}
Under the same assumptions as in 
Theorem \ref{thm1}, with in addition $n\geq3$ and $1<p<n,$ 
System \eqref{CNS} has a unique local solution $(\rho,u,\theta)$ 
with $(u,\theta)\in E_p(T)$, $\rho$ bounded away from 0 and 
$\rho-1 \in\cC([0,T];\dot B^{\frac{n}{p}}_{p,1})$.
\end{thm}

\begin{rem} Because our techniques rely on Fourier analysis, 
the same statements hold true for  periodic boundary conditions. 
\end{rem}
\begin{rem}
The equivalence between the Eulerian and the 
Lagrangian systems is provable only in the range $1<p<n$ and if $n\ge 3$
(see  Proposition \ref{prop-equiv} below), 
whence the stronger conditions on $p$ and $n.$ 
  Nevertheless the above statement  improves the results of  
\cites{danchin1,D2} as regards uniqueness : there, the condition $p\leq 2n/3$ was required. Besides,  only 
the case of small $a_0$ was considered. 

In dimension $n=2,$ or if $n\leq p<2n,$ only partial results are available. 
First, in the critical functional framework, prescribing $(a_0,u_0,\theta_0)$ or $(a_0,u_0,E_0)$ is
no longer equivalent since the product does not map $\dot B^{\frac np}_{p,1}\times\dot B^{\frac np-2}_{p,1}$ in $\dot B^{\frac np-2}_{p,1}$ any longer,
and the data are interrelated through \eqref{def-K}. 
Second, even if one chooses to work with $(a,u,E)$ rather than with $(a,u,\theta),$ having $(u,E)$ in $E_p(T)$ \emph{does not}  
quite imply that $(\bar u,\bar K)$ is in $E_p(T)$ (and the converse is false, too). Nevertheless, it is still possible 
to solve \eqref{CNSE},   see   Corollary \ref{cor1} for more details.
\end{rem}

\begin{rem}
The restriction that $1<p<n$ and $n\ge 3$ in 
Theorem \ref{thm2} is consistent with the recent paper by
Chen-Miao-Zhang \cite{chen-miao-zhang2}. 
There, the authors established the ill-posedness of 
the full compressible Navier-Stokes system in three dimension 
in the sense that the continuity of 
data-solution map fails at the origin 
in the critical Besov framework that we used, if $p>n.$ 
In other words, up to the limit case $p=n,$ Theorem \ref{thm2} is optimal 
as regards the local well-posedness issue with unknowns $(\rho,u,\theta).$ 
\end{rem}

\begin{rem}
Different formulations are known for 
expressing the third equation of \eqref{CNS}. 
Namely, the following quantities may be
used to rewrite the energy equation: the {\sl temperature} $\theta$, the {\sl total energy by unit mass} 
$M=\frac{|u|^2}{2}+\theta$ and the {\sl total energy by unit volume} 
$E=\rho(\frac{|u|^2}{2}+\theta).$ Those formulations are equivalent for smooth enough solutions. 
In the critical framework, working with the   {\sl total energy along the flow} $\overline K$ in Lagrangian coordinates allows
to get the widest range of exponents. 
\end{rem}


\subsection{Banach fixed point argument}

We end this section with a quick presentation of the Banach fixed point argument 
that will enable us to prove Theorem \ref{thm1}. 
To simplify the notation,  we drop the bars of the Lagrangian coordinates.
\medbreak
To start with, let us rewrite \eqref{CNS-l} as a system of parabolic 
equations with nonsmooth (but time independent)
coefficients. Regarding the velocity equation, 
we proceed as in \cite{danchin2}. 
Next, we write the equation for $K$ as follows:
$$\displaylines{\partial_t K
 -\div\big(k(\rho_0) \nabla (\frac{K}{\rho_0})\big)
 =\div\Big[ ( k(J^{-1}\rho_0)\adj(DX)^t\!A - k(\rho_0)\id )  \nabla\big(\frac{K}{\rho_0}\big) \hfill\cr\hfill
       +k(J^{-1}\rho_0)\adj(DX)^t\!A\nabla\big(\frac{|u|^2}{2}\big)  +\tau \cdot u -u P(J^{-1}\rho_0, K)
              \big)\Big]\cdotp }
$$
Denoting
\begin{equation}
\label{def:op}
\begin{array}{rl}
&L_{\rho_0} u 
:= \partial_t u 
    -\rho_0^{-1}\div \big(2\mu(\rho_0)D(u) + \lambda(\rho_0)\div u \id \big) \\[1ex]
\text{and}\quad
&H_{\rho_0} K 
:= \partial_t K
       -\div \big( k(\rho_0)\nabla (\rho_0^{-1}K) \big).
\end{array}
\end{equation}
System \eqref{CNS-l} thus  writes
\begin{equation}
\label{CNS-l2}
\left\{
\begin{array}{rcl}
L_{\rho_0} u + \rho_0^{-1}\nabla (\rho_0^{-1}\pi_1(\rho_0)K)
  &\!\!\!\!=\!\!\!\! &\rho_0^{-1}\div\big(I_1(u,u)\!+\!I_2(u,u)+I_3(u,u)+I_4(u,K)\big) \\[2ex]
H_{\rho_0} K &\!\!\!\!=\!\!\!\! &\div\big(I_5(u,K) + I_6(u,K)
     + I_7(u,K)+ I_8(u,u)\big),
\end{array}
\right.\end{equation} 
with
\begin{equation}\label{terms}
\begin{array}{rcl}
I_1(v,w) &\!\!\!\!:=\!\!\!\!& (\adj(DX_v)- \id)\left(2\mu(J_v^{-1}\rho_0)D_{A_v}(w) 
                         \! +\! \lambda(J_v^{-1}\rho_0)\div_{\!A_v}w\,\id\right), \\[1ex]
I_2(v,w) &\!\!\!\!:=\!\!\!\!& 2(\mu(J_v^{-1}\rho_0)-\mu(\rho_0)) D_{A_v}(w)
              \!+\! (\lambda(J_v^{-1}\rho_0)-\lambda(\rho_0))\div_{\!A_v}w\,\id, \\[1ex]
I_3(v,w) &\!\!\!\!:=\!\!\!\!& 2\mu(\rho_0)(D_{A_v}(w)-D(w))+ \lambda(\rho_0)(\div_{\!A_v}w-\div w)\id, \\[1ex]
I_4(v,\psi) &\!\!\!\!:=\!\!\!\!& -\adj(DX_v) P(J_v^{-1}\rho_0, \psi)+ \frac{\pi_1(\rho_0)}{\rho_0}\psi \id, \\[1ex]
I_5(v,\psi) &\!\!\!\!:=\!\!\!\! &( k(J_v^{-1}\rho_0)\adj(DX_v)^t\!A_v - k(\rho_0)\id )  \nabla(\frac{\psi}{\rho_0}), \\[1ex]
I_6(v,\psi) &\!\!\!\!:=\!\!\!\!& k(J_v^{-1}\rho_0)\adj(DX_v)^t\!A_v \nabla (\frac{|v|^2}{2}), \\[1ex]
I_7(v,\psi) &\!\!\!\!:=\!\!\!\!&  P(J_v^{-1}\rho_0, \psi)\adj(DX_v)\cdot v, \\[1ex]
I_8(v,w) &\!\!\!\!:=\!\!\!\!& \adj(DX_v) \big(\lambda(J_v^{-1}\rho_0)\div_{\!A_v} w \ \id 
                    + 2\mu(J_v^{-1}\rho_0) D_{A_v}(w) \big) \cdot w.  
\end{array}
\end{equation}
In order to solve \eqref{CNS-l} locally, 
it suffices to show that the map
\begin{equation}
\label{map}
\Phi : (v,\psi) \mapsto (u,K)
\end{equation}
with $(u,K)$ the solution to 
\begin{equation}
\label{CNS-lf}
\left\{
\begin{array}{rcl}
L_{\rho_0} u  + \rho_0^{-1}\nabla(\rho_0^{-1}\pi_1(\rho_0)K)
   &\!\!\!=\!\!\!& \rho_0^{-1}\div\!\big(I_1(v,v)+I_2(v,v)+I_3(v,v)+I_4(v,\psi)\big) \\[2ex]
H_{\rho_0} K &\!\!\!=\!\!\!& \div\big(I_5(v,\psi) + I_6(v,\psi)
     + I_7(v,\psi)+ I_8(v,v)\big),
\end{array}
\right.
\end{equation} 
has a fixed point in $E_p(T)$ for small enough $T$.
\bigbreak
The rest of the paper unfolds as follows: in the second section, we establish  the 
maximal regularity estimates for the linear  parabolic system corresponding to the l.h.s. of \eqref{CNS-lf}. 
It turns out that completely  decoupling the system into two parabolic equations 
for the velocity and energy will cause some loss of estimate: 
we ought to take into account the pressure term as the linear 
term of the system, which is not necessary in the 
barotropic case. 
In the third section, we shall prove Theorem \ref{thm1} and Theorem \ref{thm2} 
by combining the a priori estimate in the second section and Banach's fixed point theorem. 
In the Appendix, we list some results concerning the Lagrangian coordinates and Besov spaces that may be
found in the literature (see \cites{bahouri-chemin-danchin,danchin2,danchin-mucha}).


\section{A priori estimates for linear parabolic systems}
\label{s:parabolic}
We here aim at  establishing well-posedness and a priori estimates for 
the linear part of \eqref{CNS-lf}, namely
\begin{equation}
\label{L-LMK}
\left\{
\begin{array}{l}
\partial_t u 
    -\rho_0^{-1}\div \big(2\mu(\rho_0)D(u) + \lambda(\rho_0)\div u \,\id \big)
    + \rho_0^{-1}\nabla (\rho_0^{-1}\pi_1(\rho_0)K)
  = f, \\[2ex]
\partial_t K
       -\div \big(k(\rho_0)\nabla(\rho_0^{-1} K)\big)= g. 
\end{array}
\right.
\end{equation}

The analysis of the first equation is based on results that have been established recently in \cite{danchin2} for 
the following Lam\'e system with 
nonsmooth coefficients:
\begin{equation}\label{LM}
\partial_t u - 2a\div(\mu D(u))-b\nabla(\lambda\div u) = f,
\end{equation}
(here both $u$ and $f$ are valued in $\R^n$) when 
 the following uniform ellipticity condition is satisfied:
\begin{equation}\label{ellip}
\alpha:=\min\left( \inf_{(t,x)\in[0,T]\times\R^n}(a\mu)(t,x), \ 
                  \inf_{(t,x)\in[0,T]\times\R^n}(2a\mu+b\lambda)(t,x)\right) > 0.
\end{equation}
\begin{prop}[\cite{danchin2}]
\label{apriori:LM} 
Let $a$, $b$, $\lambda$ and $\mu$ be bounded 
functions satisfying \eqref{ellip}.
Assume that $a\nabla\mu$, $b\nabla\lambda$, $\mu\nabla a$ and 
$\lambda\nabla b$ are in $L^\infty(0,T;\dot B^{\frac{n}{p}-1}_{p,1})$
for some $1<p<2n,$ and  that there exist some constants $\bar a$,
$\bar b$, $\bar \lambda$ and $\bar\mu$ satisfying
\begin{equation}
\nonumber
2\bar a\bar\mu+\bar b\bar\lambda>0 \ \text{and} \ \bar a\bar\mu>0,
\end{equation}
and such that $a-\bar a$, $b-\bar b$, $\lambda-\bar\lambda$ and $\mu-\bar\mu$
are in $\cC([0,T];\dot B^{\frac{n}{p}}_{p,1})$. Finally, suppose that 
\begin{equation}
\nonumber
\lim_{m\rightarrow+\infty}
\|(\id-\dot S_m)(a\nabla\mu, b\nabla\lambda, 
   \mu\nabla a, \lambda\nabla b)\|_{L^\infty_T(\dot B^{\frac{n}{p}-1}_{p,1})}
=0.
\end{equation}
Then for any data $u_0\in\dot B^{\frac np-1}_{p,1}$ and $f\in L^1(0,T;\dot B^{\frac np-1}_{p,1})$, System \eqref{LM} admits a unique solution 
$u\in\cC([0,T];\dot B^{\frac{n}{p}}_{p,1})$ with $\nabla u\in L^1(0,T;\dot B^{\frac np}_{p,1})$.
\smallbreak
Furthermore, there exist two constants $\eta$ and $C$ such 
that if $m$ is so large as to satisfy
\begin{equation}
\label{cond:LM1}
\min\left( \inf_{(t,x)\in[0,T]\times\R^n}\dot S_m (a\mu)(t,x), \ 
           \inf_{(t,x)\in[0,T]\times\R^n}\dot S_m (2a\mu+b\lambda)(t,x)\right) 
      \ge \frac{\alpha}{2},
\end{equation}
\begin{equation}
\label{cond:LM2}
\|(\id-\dot S_m)(a\nabla\mu, b\nabla\lambda, \mu\nabla a, \lambda\nabla b)\|_{L^\infty_T(\dot B^{\frac{n}{p}-1}_{p,1})}\le\eta\alpha,
\end{equation}
then we have for all $t\in[0,T]$,
$$\displaylines{
\|u\|_{L^\infty_t(\dot B^{\frac np-1}_{p,1})}+\alpha \|\nabla u\|_{L^1_t(\dot B^{\frac np}_{p,1})} \hfill\cr\hfill
\le C (\|u_0\|_{\dot B^{\frac np-1}_{p,1}}
    +\|f\|_{L^1_t(\dot B^{\frac np-1}_{p,1})}) 
\exp\left(\frac{C}{\alpha}
   \int_0^t\|\dot S_m(a\nabla\mu, b\nabla\lambda,
         \mu\nabla a, \lambda\nabla b)\|_{\dot B^{\frac{n}{p}}_{p,1}}^2 d\tau \right).}
$$
\end{prop}

As the energy equation of \eqref{L-LMK}  is of the following form: 
\begin{equation}
\label{K}
\partial_t u - \div(k \nabla (cu)) = f,
\end{equation}
and thus does not quite enter in the framework of Proposition \ref{apriori:LM}, we shall need the following statement.
\begin{prop}\label{apriori:K}
Let $k$ be a bounded function such that there 
exists a constant $\beta$ with $k\ge \beta>0$.
Assume that $\nabla k$ and $\nabla c$ are
in $L^\infty(0,T;\dot B^{\frac{n}{p}-1}_{p,1})$
for some $1<p<\infty,$ that 
$$\lim_{m\rightarrow+\infty}
\|(\id-\dot S_m)(k\nabla c, c\nabla k)\|_{L^\infty_T(\dot B^{\frac{n}{p}-1}_{p,1})}=0,
$$
and that $k-\bar k$ and $c-\bar c$ are in $\cC([0,T];\dot B^{\frac np-1}_{p,1})$ for
some positive constants $\bar k$ and $\bar c.$
\medbreak
Then there exist two constants $\eta$ and $C$ such 
that if for some $m\in\Z$ we have
\begin{equation}
\label{cond:K1}
\inf_{(t,x)\in[0,T]\times\R^n}\dot S_m (kc)(t,x) \ge \frac{\beta}{2},
\end{equation}
\begin{equation}
\label{cond:K2}
\|(\id-\dot S_m) (k\nabla c, c\nabla k)\|_{L^\infty_T(\dot B^{\frac{n}{p}-1}_{p,1})}
\le\eta\beta,
\end{equation}
then the solutions to \eqref{LM} satisfy for all $t\in[0,T]$,
\begin{equation}
\notag
\begin{split}
\|u\|_{L^\infty_t(\dot B^s_{p,1})}
&+\beta \|u\|_{L^1_t(\dot B^{s+2}_{p,1})} \\
&\le C (\|u_0\|_{\dot B^s_{p,1}}+\|f\|_{L^1_t(\dot B^s_{p,1})}) 
\exp\left(\frac{C}{\beta}
   \int_0^t\|\dot S_m(k\nabla c,c\nabla k)\|_{\dot B^{\frac{n}{p}}_{p,1}}^2 
   d\tau   \right)
\end{split}
\end{equation}
whenever $s$ satisfies
\begin{equation}
\label{cond:sK}
\displaystyle -\min\Bigl(\frac{n}{p},\frac{n}{p'}\Bigr)-1<s\le\frac{n}{p}-2.
\end{equation}
\end{prop}

\begin{proof} We focus on the proof of a priori estimates. Existence follows from the continuity method
as for Proposition \ref{apriori:LM} (see \cite{danchin2}). 

First, we smooth out the coefficient $kc$ according to the low
frequency cut-off operator $\dot S_m,$ with  $m\in\Z$ to be determined later:
\begin{equation}\nonumber
\partial_t u - \div(\dot S_m (kc) \nabla u) = f + \div(k\nabla c\cdot u)
     + \div \big ((\id-\dot S_m) (kc) \nabla u \big).
\end{equation}
Next, applying  Littlewood-Paley operator $\dot\Delta_j$ to the
above equation yields \begin{equation}\nonumber
\begin{split}
&\partial_t u_j - \div(\dot S_m (kc) \nabla u_j )
   = f_j + \div \dot\Delta_j ( \dot S_m (k \nabla c) \cdot u) 
         + \div \dot\Delta_j ( (\id-\dot S_m)(k\nabla c) \cdot u ) \\
&\qquad\qquad\qquad\qquad\qquad\quad
     + \div [\dot\Delta_j, \dot S_m (kc)] \nabla u
     + \div \dot\Delta_j ( (\id-\dot S_m )(kc) \nabla u ).
\end{split}
\end{equation}
{}From energy arguments combined with  the  Bernstein-type inequality of the Appendix of \cite{danchin1}, we get (formally) 
$$\displaylines{
\frac{d}{dt} \|u_j\|_{L^p} 
 + \beta 2^{2j} \|u_j\|_{L^p} 
\lesssim \|f_j\|_{L^p}  
  +  \|\div \dot\Delta_j ( \dot S_m (k \nabla c) \cdot u)\|_{L^p} \hfill\cr\hfill
  +  \|\div \dot\Delta_j ( (\id-\dot S_m)(k\nabla c) \cdot u )\|_{L^p}+ \|\div [\dot\Delta_j, \dot S_m (kc)] \nabla u\|_{L^p}
     + \|\div \dot\Delta_j ( (\id-\dot S_m )(kc) \nabla u )\|_{L^p}.}
$$
Whence, multiplying both sides by $2^{js}$ and performing a $\ell^1$ summation over $j\in\Z,$
\begin{multline}\label{eq:star}
\|u\|_{L_t^\infty(\dot B^s_{p,1})}+\beta\|u\|_{L_t^1(\dot B^{s+2}_{p,1})} \lesssim \|u_0\|_{\dot B^s_{p,1}}+\|f\|_{L_t^1(\dot B^{s}_{p,1})}\\
+\int_0^t \sum_j2^{js}\bigl(\|\div \dot\Delta_j ( (\id-\dot S_m )(kc) \nabla u )\|_{L^p}+ \|\div \dot\Delta_j ((\id- \dot S_m) (k \nabla c) \cdot u)\|_{L^p}\bigr)\,d\tau\\
+\int_0^t\sum_j 2^{js}\bigl( \|\div \dot\Delta_j(\dot S_m(k\nabla c) \cdot u )\|_{L^p}+ \|\div [\dot\Delta_j, \dot S_m (kc)] \nabla u\|_{L^p}
     \bigr)\,d\tau.
\end{multline}
In the following computations, let us denote by $(c_j)_{j\in\Z}$ a sequence belonging 
to the unit sphere of $\ell^1(\Z).$ 
If $\displaystyle -\min\Big(\frac{n}{p},\frac{n}{p'}\Big)-1<s\le\frac{n}{p}-1$ then we have by Proposition \ref{prop:prd}:
\begin{equation}\nonumber
\begin{split}
\|\div \dot\Delta_j ( (\id-\dot S_m )(k c) \nabla u )\|_{L^p}
\le c_j 2^{-js} 
  \| (\id-\dot S_m)(kc)\|_{\dot B^{\frac{n}{p}}_{p,1}}
  \| \nabla u \|_{\dot B^{s+1}_{p,1}}.
\end{split}
\end{equation}
If $s$ satisfies $\displaystyle -\min\Big(\frac{n}{p},\frac{n}{p'}\Big)-1<s\le\frac{n}{p}-2$ then
\begin{equation}\nonumber
\begin{split}
\|\div\dot\Delta_j ((\id-\dot S_m)(k\nabla c) \cdot u)\|_{L^p}
\le c_j 2^{-js} 
  \| (\id-\dot S_m)(k\nabla c)\|_{\dot B^{\frac{n}{p}-1}_{p,1}}
  \|u\|_{\dot B^{s+2}_{p,1}}.
\end{split}
\end{equation}
Consequently, the second line of the \eqref{eq:star} may be absorbed by the l.h.s. if $\eta$ has been chosen small enough
in \eqref{cond:K2}. 
Next, if $s$ satisfies $\displaystyle -\min\Big(\frac{n}{p},\frac{n}{p'}\Big)-1<s\le\frac{n}{p}-1$ then
\begin{equation}\nonumber
\begin{split}
\|\div \dot\Delta_j ( \dot S_m (k \nabla c) \cdot u)\|_{L^p}
\le c_j 2^{-js} 
  \|\dot S_m (k \nabla c) \|_{\dot B^{\frac{n}{p}}_{p,1}}
  \|u\|_{\dot B^{s+1}_{p,1}},
\end{split}
\end{equation}
Finally, for $\displaystyle -\min\Big(\frac{n}{p},\frac{n}{p'}\Big)-1<s\le\frac{n}{p}-1,$ 
we  have by Proposition \ref{prop:comm1}, 
\begin{equation}\nonumber
\begin{split}
 \|\div [\dot\Delta_j, \dot S_m (kc)] \nabla u\|_{L^p}
\le c_j 2^{-js} 
  \| \nabla \dot S_m (kc)\|_{\dot B^{\frac{n}{p}}_{p,1}}
  \| \nabla u \|_{\dot B^{s}_{p,1}}.
\end{split}
\end{equation}
 Therefore by interpolation and Young's inequality, we get for all $\eta>0,$
\begin{equation}\nonumber
\begin{split}
 \|\div [\dot\Delta_j, &\dot S_m (kc)] \nabla u\|_{L^p}
+ \|\div \dot\Delta_j ( \dot S_m (k \nabla c) \cdot u)\|_{L^p} \\
&\le c_j 2^{-js} 
  \bigg(\frac C{\eta\beta}   (\| \nabla \dot S_m (kc)\|_{\dot B^{\frac{n}{p}}_{p,1}}^2
    +\| \dot S_m (k \nabla c)\|_{\dot B^{\frac{n}{p}}_{p,1}}^2)
     \|u\|_{\dot B^{s}_{p,1}}
     +\eta\beta\|u\|_{\dot B^{s+2}_{p,1}}\bigg)\cdotp
\end{split}
\end{equation}
It is now clear that taking $\eta$  small enough completes the proof of the proposition.
\end{proof}

Combining Propositions \ref{apriori:LM} and \ref{apriori:K}, one can 
now consider  the following linear system: 
\begin{equation}
\label{LM-K}
\left\{\begin{array}{l}
\partial_t u 
    -a\div \big( 2\mu D(u) + \lambda \div u \id
                                       -\pi K \id \big)= f,  \\[1ex]
\partial_t K - \div(k \nabla(c K)) = g,
\end{array}\right.
\end{equation}

\begin{prop}
\label{apriori:LM-K}
Let $1<p<2n.$ Let  $u_0\in\dot B^{\frac np-1}_{p,1},$ $K_0\in\dot B^{\frac np-2}_{p,1},$ $f\in L^1(0,T;\dot B^{\frac np-1}_{p,1})$
and $g\in L^1(0,T;\dot B^{\frac np-2}_{p,1}).$
Let $a$, $b$, $\lambda$ and $\mu$ satisfy the 
assumptions of Proposition $\ref{apriori:LM}$ 
and $k$ and $c$ satisfy those of Proposition $\ref{apriori:K}$ with $s=\frac np-2.$ 
Assume that $\pi$ belongs to the multiplier space
\footnote{
The multipler space $\mathcal{M}(\dot B^{s}_{p,1})$ is 
the set of all functions 
$f\in \dot B^{s}_{p,1}$ 
such that 
$\displaystyle 
\|f\|_{\mathcal{M}(\dot B^{s}_{p,1})}
:=\sup_{\|h\|_{\dot B^{s}_{p,1}}=1} 
         \|hf\|_{\dot B^{s}_{p,1}}<\infty$.} 
$\mathcal{M}(\dot B^{\frac np}_{p,1}).$  Finally, suppose that
$$
\lim_{m\rightarrow+\infty}
\|(\id-\dot S_m)(k\nabla c, c\nabla k,a\nabla\mu,b\nabla\lambda, 
   \mu\nabla a, \lambda\nabla b)\|_{L^\infty_T(\dot B^{\frac{n}{p}-1}_{p,1})}=0.
$$
Then System \eqref{LM-K} admits a unique solution $(u,K)$ with 
$$
u\in\cC([0,T];\dot B^{\frac np-1}_{p,1})\cap L^1(0,T;\dot B^{\frac np+1}_{p,1})\quad\hbox{and}\quad
K\in\cC([0,T];\dot B^{\frac np-2}_{p,1})\cap L^1(0,T;\dot B^{\frac np}_{p,1}).
$$
Besides, if $m$ is large enough (as in Propositions $\ref{apriori:LM}$ and $\ref{apriori:K}$)
then $(u,K)$ fulfills for all $t\in[0,T]$,
$$\displaylines{
\|K\|_{L^\infty_t(\dot B^{\frac np-2}_{p,1})}+\beta\|K\|_{L^1_t(\dot B^{\frac np}_{p,1})} \leq 
C\biggl(\|K_0\|_{\dot B^{\frac np-2}_{p,1}}+ \|g\|_{L^1_t(\dot B^{\frac np}_{p,1})}\biggr)\hfill\cr\hfill\times
\exp\biggl(\frac C\beta\int_0^t\|\dot S_m(k\nabla c,c\nabla k)\|_{\dot B^{\frac np}_{p,1}}^2\biggr),\cr
\|u\|_{L^\infty_t(\dot B^{\frac np-1}_{p,1})}+\alpha\|u\|_{L^1_t(\dot B^{\frac np+1}_{p,1})}\leq 
C\biggl(\|u_0\|_{\dot B^{\frac np-1}_{p,1}} +\|f\|_{L^1_t(\dot B^{\frac np-1}_{p,1})}
\hfill\cr\hfill+\|a\|_{\cM(\dot B^{\frac np-1}_{p,1})}\|\pi\|_{\cM(\dot B^{\frac np}_{p,1})}\|K\|_{L^1_t(\dot B^{\frac np}_{p,1})}\biggr)
\exp\biggl(\frac C\alpha\int_0^t\|\dot S_m(a\nabla\mu,b\nabla\lambda, 
   \mu\nabla a, \lambda\nabla b)\|_{\dot B^{\frac np}_{p,1}}^2\biggr)\cdotp}
$$
\end{prop}
\begin{proof}
It suffices to first solve the second equation of \eqref{LM-K} according to Proposition \ref{apriori:K}, 
then look at $u$  as the solution to 
$$
\partial_tu-a\div(2\mu D(u)+\lambda\div u\id)=f-a\nabla(\pi K).
$$
Given the assumptions on $a$ and $\pi,$ and the fact that $K$ is in $L^1(0,T;\dot B^{\frac np}_{p,1}),$
we see that $u$ may be constructed according to Proposition \ref{apriori:LM}.
\end{proof}


\section{Proof of the main theorem}\label{s:main}

Let  $(u_L,K_L)$ be the solution  
to the linear system corresponding to the l.h.s. of \eqref{CNS-l2} with $\rho=1,$ namely 
\begin{equation*}
\begin{split}
L_1 u_L + \pi_1(1)\nabla K_L = 0, \quad u_L|_{t=0} = u_0, \\
H_1 K_L = 0, \quad  K_L|_{t=0} = K_0.
\end{split}
\end{equation*}

\subsection{The fixed point scheme}

We claim that the Banach fixed point
theorem applies to the map $\Phi$ defined in 
\eqref{map} in some closed ball $\bar B_{E_p(T)}((u_L,\theta_L),R)$
with suitably small $T$ and $R$. 
\medbreak
To justify our claim, we set  $\tilde u := u-u_L$ and $\tilde K := K-K_L,$ and observe that solving \eqref{CNS-lf} for some given  $(v,S)\in E_p(T)$
is equivalent to  solving  
\begin{equation*}
\left\{
\begin{array}{l}
L_{\rho_0} \tilde u + \rho_0^{-1}\nabla( \rho_0^{-1} \pi_1(\rho_0)\tilde K ) 
= \rho_0^{-1}\div (I_1(v,v)+I_2(v,v)+I_3(v,v)+I_4(v,\psi)) \\[1ex]
\qquad\qquad\qquad\qquad\qquad\qquad
    + (L_1-L_{\rho_0})u_L
    - \rho_0^{-1}\nabla(\rho_0^{-1}\pi_1(\rho_0)K_L)+\nabla(\pi_1(1)K_L), \\[2ex]
H_{\rho_0} \tilde K = \div \big(I_5(v,\psi) + I_6(v,\psi)
      + I_7(v,\psi)+ I_8(v,v) \big) + (H_1-H_{\rho_0})K_L.
\end{array}
\right.
\end{equation*} 
{}From  the definition of the space $\dot B^{\frac{n}{p}}_{p,1}$ (which involves a convergent series)
and the fact that it embeds  in the set of bounded continuous functions, it is clear that there exists some $m\in\Z$ so that
$$\displaylines{
\min\left( \inf_{x\in\R^n} \dot S_m (\frac{\mu(\rho_0)}{\rho_0}), \  
  \inf_{x\in\R^n}
  \dot S_m (2\frac{\mu(\rho_0)}{\rho_0}+\frac{\lambda(\rho_0)}{\rho_0}), \ 
  \inf_{x\in\R^n} \dot S_m (\frac{k(\rho_0)}{\rho_0})
 \right) \ge \frac{\max(\alpha,\beta)}{2},\cr
\bigl\|(\id-\dot S_m)\bigl(\frac{\mu(\rho_0)}{\rho_0^2}\nabla\rho_0, 
  \frac{\mu'(\rho_0)}{\rho_0}\nabla\rho_0, 
  \frac{\lambda(\rho_0)}{\rho_0^2}\nabla\rho_0, 
  \frac{\lambda'(\rho_0)}{\rho_0}\nabla\rho_0,
  \frac{k(\rho_0)}{\rho_0}\nabla\rho_0\bigr)\bigr\|_{L^\infty_T(\dot B^{\frac{n}{p}-1}_{p,1})}\hfill]\cr\hfill\le\eta\min(\alpha,\beta).}
$$
Therefore, in order to solve the above system by means of  Proposition \ref{apriori:LM-K},
it suffices to check that  the r.h.s.  of the first and  second equations are in $L^1(0,T;\dot B^{\frac{n}{p}-1}_{p,1})$ 
and $L^1(0,T;\dot B^{\frac{n}{p}-2}_{p,1})$, respectively.

\medbreak\noindent
{\sl First step: Stability of the ball 
$\bar B_{E_p(T)}((u_L,K_L),R)$ for suitably small $T$ and 
$R$. } \smallbreak
{}From now on, we assume that for a small enough $\tilde c,$ we have
\begin{equation}\label{eq:smallv}
\|Dv\|_{L^1_T(\dot B^{\frac{n}{p}}_{p,1})}\le\tilde c.
\end{equation}
Proposition \ref{apriori:LM-K} and the definition of 
the multiplier space $\mathcal{M}(\dot B^{\frac{n}{p}-1}_{p,1})$ 
ensure that 
\begin{equation}
\label{pr:apriori:LM-K}
\begin{split}
\|(\tilde u,\tilde K)\|_{E_p(T)} 
&\le C e^{C_{\rho_0,m}T}
\Big( \|(L_1-L_{\rho_0})u_L\|_{L^1_T(\dot B^{\frac{n}{p}-1}_{p,1})}
  +\|(H_1-H_{\rho_0})K_L\|_{L^1_T(\dot B^{\frac{n}{p}-2}_{p,1})} \\ 
  & +\|(\rho_0^{-1}\pi_1(\rho_0)-\pi_1(1))K_L
        \|_{L^1_T(\dot B^{\frac{n}{p}}_{p,1})}     \\
  & +\|\rho_0^{-1}\|_{\mathcal{M}(\dot B^{\frac{n}{p}-1}_{p,1})} 
     \|I_1(v,v)+I_2(v,v)+I_3(v,v)+I_4(v,\psi)\|_{L^1_T(\dot B^{\frac{n}{p}}_{p,1})} \\
  & + \|I_5(v,\psi) + I_6(v,\psi)
      + I_7(v,\psi)+ I_8(v,v) 
      \|_{L^1_T(\dot B^{\frac{n}{p}-1}_{p,1})} \Big). \\
\end{split}
\end{equation}
We may confirm that 
$\rho_0^{-1}$ belongs to $\mathcal{M}(\dot B^{\frac{n}{p}-1}_{p,1})$
by the product estimate:
\begin{equation}
\notag
\begin{split}
\|\rho_0^{-1} h \|_{\dot B^{\frac{n}{p}-1}_{p,1}} 
\le \|(\frac{a_0}{1+a_0}-1)h\|_{\dot B^{\frac{n}{p}-1}_{p,1}}
\le (\|a_0\|_{\dot B^{\frac{n}{p}}_{p,1}}+1)\|h\|_{\dot B^{\frac{n}{p}-1}_{p,1}}.
\end{split}
\end{equation}
Likewise, 
\begin{eqnarray}\label{L1}
&\|(L_1-L_{\rho_0})u_L\|_{L^1_T(\dot B^{\frac{n}{p}-1}_{p,1})}
\le (\|a_0\|_{\dot B^{\frac{n}{p}}_{p,1}}+1)
    \|a_0\|_{\dot B^{\frac{n}{p}}_{p,1}} 
    \|Du_L\|_{L^1_T(\dot B^{\frac{n}{p}}_{p,1})},\\\label{L2}
&\|(\rho_0^{-1}\pi_1(\rho_0)-\pi_1(1)K_L)\|_{L^1_T(\dot B^{\frac{n}{p}-1}_{p,1})} \le 
    \|a_0\|_{\dot B^{\frac{n}{p}}_{p,1}}\|K_L\|_{L^1_T(\dot B^{\frac{n}{p}}_{p,1})}\\\label{L3}
\hbox{and}\quad
&\|(H_1-H_{\rho_0})K_L\|_{L^1_T(\dot B^{\frac{n}{p}-2}_{p,1})} 
\lesssim     (\|a_0\|_{\dot B^{\frac{n}{p}}_{p,1}}+1)^2
    \|a_0\|_{\dot B^{\frac{n}{p}}_{p,1}} 
    \|K_L\|_{L^1_T(\dot B^{\frac{n}{p}}_{p,1})}. 
\end{eqnarray}

\noindent{\sl \underbar{Estimate of $I_1$, $I_2$, $I_3$}}~:
Terms $I_1$, $I_2$ and $I_3$ have been estimated in \cite{danchin2} as follows : 
\begin{equation*}
\|I_j(v,w)\|_{L^1_T(\dot B^{\frac{n}{p}}_{p,1})} 
\lesssim (\|a_0\|_{\dot B^{\frac{n}{p}}_{p,1}}+1) 
 \|Dv\|_{L^1_T(\dot B^{\frac{n}{p}}_{p,1})} \|Dw\|_{L^1_T(\dot B^{\frac{n}{p}}_{p,1})}\quad\hbox{for}\quad
 j=1,2,3.
\end{equation*}

\noindent{\sl \underbar{Estimate of $I_4$}}~:
Let us recall that the pressure is given by
\begin{equation}
\nonumber
P(J_v^{-1} \rho_0, \psi) 
= \pi_0(J_v^{-1} \rho_0)
 + \Bigl(\frac{1}{\rho_0} \psi
      - \frac{|v|^2}{2}\Bigr) \pi_1(J_v^{-1} \rho_0)
\end{equation}
so that $I_4$ can be written as 
$$
\displaylines{\quad I_4(v,\psi) = -\adj(DX_v) \pi_0(J_v^{-1}\rho_0) 
-\Bigl(\adj(DX_v)\frac{\pi_1(J_v^{-1}\rho_0)}{\rho_0}
                -\frac{\pi_1(\rho_0)}{\rho_0} \id\Bigr)\psi\hfill\cr\hfill
 +\adj(DX_v) \pi_1(J_v^{-1}\rho_0)\frac{|v|^2}{2}\cdotp\quad}
 $$
Let us notice that 
\begin{equation}\nonumber
J_v^{-1}\rho_0 -1 = (J_v^{-1}-1)(a_0+1)+a_0.
\end{equation}
Hence, taking advantage of \eqref{eq:smallv} and of the results of the appendix,
\begin{equation}\nonumber
\|J_v^{-1}\rho_0 -1
    \|_{L^\infty_T(\dot B^{\frac{n}{p}}_{p,1})}
\lesssim 
(\|a_0\|_{\dot B^{\frac{n}{p}}_{p,1}}+1).
\end{equation}
Next, we have
\begin{equation}\nonumber
\frac{\pi_1(J_v^{-1}\rho_0)}{\rho_0} = 
(\pi_1(J_v^{-1}\rho_0)-\pi_1(1)+\pi_1(1))(1-\frac{a_0}{a_0+1}). 
\end{equation}
Therefore, using again \eqref{eq:smallv} together with composition estimates yields
$$
\|\frac{\pi_1(J_v^{-1}\rho_0)}{\rho_0}
    \|_{\mathcal{M} (L^\infty_T(\dot B^{\frac{n}{p}}_{p,1}) )}
\lesssim (\|a_0\|_{\dot B^{\frac{n}{p}}_{p,1}}+1)^2. $$
Since 
$$\displaylines{\quad
\Bigl(\adj(DX_v)\frac{\pi_1(J_v^{-1}\rho_0)}{\rho_0}
                -\frac{\pi_1(\rho_0)}{\rho_0} \id\Bigr)\psi \hfill\cr\hfill
= (\adj(DX_v)-\id) \frac{\pi_1(J_v^{-1}\rho_0)}{\rho_0} \psi \\
 + \Bigl(\frac{\pi_1(J_v^{-1}\rho_0)}{\rho_0} -\frac{\pi_1(\rho_0)}{\rho_0}\Bigr) \psi \id\quad}
$$
and 
\begin{equation}\nonumber
\frac{\pi_1(J_v^{-1}\rho_0)}{\rho_0} 
- \frac{\pi_1(\rho_0)}{\rho_0}
=(\pi_1(J_v^{-1}\rho_0)-\pi_1(\rho_0)) (1-\frac{a_0}{a_0+1}), 
\end{equation}
we conclude that 
$$
\|(\adj(DX_v)\frac{\pi_1(J_v^{-1}\rho_0)}{\rho_0}
   -\frac{\pi_1(\rho_0)}{\rho_0}) \psi \id \|_{L^1_T(\dot B^{\frac{n}{p}}_{p,1})} \lesssim 
(\|a_0\|_{\dot B^{\frac{n}{p}}_{p,1}}+1)^2
\|Dv\|_{L^1_T(\dot B^{\frac{n}{p}}_{p,1})}
\|\psi\|_{L^1_T(\dot B^{\frac{n}{p}}_{p,1})}.
$$
Finally, by Proposition \ref{prop:prd} and \ref{prop:comp}. 
$$
\|\frac{|v|^2}{2}\pi_1(J_v^{-1}\rho_0)\|_{L^1_T(\dot B^{\frac{n}{p}}_{p,1})}
\lesssim (1+\|a_0\|_{\dot B^{\frac{n}{p}}_{p,1}})\|v\|_{L^2_T(\dot B^{\frac{n}{p}}_{p,1})}^2.
$$
Therefore, using the hypothesis that $\pi_0(1)=0$, we have
\begin{equation*}
\begin{split}
\|I_4(v,\psi)&\|_{L^1_T(\dot B^{\frac{n}{p}}_{p,1})} 
\lesssim 
T \|\pi_0(J_v^{-1}\rho_0) \|_{L^\infty_T(\dot B^{\frac{n}{p}}_{p,1})}  
   \!+\!\|(\adj(DX_v)\frac{\pi_1(J_v^{-1}\rho_0)}{\rho_0}
   -\frac{\pi_1(\rho_0)}{\rho_0}\id) \psi  \|_{L^1_T(\dot B^{\frac{n}{p}}_{p,1})} \\
&\qquad\qquad\qquad\qquad\qquad\qquad
   +\|\frac{|v|^2}{2}\pi_1(J_v^{-1}\rho_0)\|_{L^1_T(\dot B^{\frac{n}{p}}_{p,1})} \\
&\lesssim 
(\|a_0\|_{\dot B^{\frac{n}{p}}_{p,1}}+1)
\big(T + (\|a_0\|_{\dot B^{\frac{n}{p}}_{p,1}}+1)
          \|\psi\|_{L^1_T(\dot B^{\frac{n}{p}}_{p,1})} 
          \|Dv\|_{L^1_T(\dot B^{\frac{n}{p}}_{p,1})} +\|v\|_{L^2_T(\dot B^{\frac{n}{p}}_{p,1})}^2)  \big). 
\end{split}
\end{equation*}

\noindent{\sl \underbar{Estimate of $I_5$}}~:
We can write the term $I_5$ as 
\begin{equation*}
\begin{split}
I_5(v,\psi) 
&= ( k(J^{-1}\rho_0)\adj(DX_v){}^t\!A_v 
           - k(\rho_0)\id ) \nabla (\frac{\psi}{\rho_0})\\
&= \big( ( k(J_v^{-1}\rho_0)\!-\!k(\rho_0) )\bigl(\id\!+\!(\adj(DX_v){}^t\!A_v\!-\!\id)\bigr) 
         \!+\! k(\rho_0)(\adj(DX_v){}^t\!A_v\! -\!\id) \big) \nabla(\frac{\psi}{\rho_0})\cdotp
\end{split}
\end{equation*}
Note that 
\begin{equation}\nonumber
\adj(DX_v)^t\!A_v -\id = (\adj(DX_v)-\id)(^t\!A_v-\id)+(\adj(DX_v)-\id)+(^t\!A_v-\id),
\end{equation}
and hence, we have according to Proposition \ref{prop:flow1}
\begin{equation}\nonumber
\begin{split}
\|\adj(DX_v)^t\!A_v -\id\|_{L^\infty_T(\dot B^{\frac{n}{p}}_{p,1})} 
&\lesssim \|\adj(DX_v)-\id\|_{L^\infty_T(\dot B^{\frac{n}{p}}_{p,1})}
   \|^t\!A_v-\id\|_{L^\infty_T(\dot B^{\frac{n}{p}}_{p,1})} \\
&\qquad\qquad
   +\|\adj(DX_v)-\id\|_{L^\infty_T(\dot B^{\frac{n}{p}}_{p,1})}
   +\|^t\!A_v-\id\|_{L^\infty_T(\dot B^{\frac{n}{p}}_{p,1})} \\
&\lesssim \|Dv\|_{L^1_T(\dot B^{\frac{n}{p}}_{p,1})}^2
   +2\|Dv\|_{L^1_T(\dot B^{\frac{n}{p}}_{p,1})}\\
&\lesssim \|Dv\|_{L^1_T(\dot B^{\frac{n}{p}}_{p,1})}.
\end{split}
\end{equation}
Next, we have 
\begin{equation*}
\begin{split}
k(J_v^{-1}\rho_0)-k(\rho_0)
&= \int_0^1 k'((J_v^{-1}-1)a_0\tau+a_0+1) d\tau \times (J_v^{-1}-1)a_0.\\
\end{split}
\end{equation*}
Hence thanks to Propositions \ref{prop:prd} and \ref{prop:comp}, and to \eqref{eq:smallv}, we have
\begin{equation*}
\begin{split}
\|k(J_v^{-1}\rho_0)-k(\rho_0)\|_{L^\infty_T(\dot B^{\frac{n}{p}}_{p,1})}
&\lesssim
(\|(J_v^{-1}-1)a_0+a_0\|_{L^\infty_T(\dot B^{\frac{n}{p}}_{p,1})}+1)
\|(J_v^{-1}-1)a_0\|_{L^\infty_T(\dot B^{\frac{n}{p}}_{p,1})} \\
&\lesssim(\|a_0\|_{\dot B^{\frac{n}{p}}_{p,1}}+1)
\|a_0\|_{\dot B^{\frac{n}{p}}_{p,1}}
\|Dv\|_{L^1_T(\dot B^{\frac{n}{p}}_{p,1})}.
\end{split}
\end{equation*}
Note also that
$$
\|k(\rho_0)\|_{\cM(\dot B^{\frac{n}{p}}_{p,1})}
\lesssim\|a_0\|_{\dot B^{\frac{n}{p}}_{p,1}}+1.
$$
Therefore, 
\begin{equation*}
\begin{array}{rl}
\|&\!\!\!\!\!I_5(v,\psi) \|_{L^1_T(\dot B^{\frac{n}{p}-1}_{p,1})} \\
&\lesssim 
\| k(J^{-1}\rho_0)-k(\rho_0) \adj(DX_v)^t\!A_v\nabla(\frac{\psi}{\rho_0})
               \|_{L^1_T(\dot B^{\frac{n}{p}-1}_{p,1})} \\
& \qquad 
+\|(\adj(DX_v)^t\!A_v-\id) k(\rho_0) \nabla(\frac{\psi}{\rho_0})
            \|_{L^1_T(\dot B^{\frac{n}{p}-1}_{p,1})} \\
&\lesssim 
(\|a_0\|_{\dot B^{\frac{n}{p}}_{p,1})}\!+\!1)
\|k(J^{-1}\rho_0)-k(\rho_0)\|_{L^\infty_T(\dot B^{\frac{n}{p}}_{p,1})}
  \|\adj(DX_v)^t\!A_v\|_{L^\infty_T(\dot B^{\frac{n}{p}}_{p,1})}
  \| \nabla(\frac{\psi}{\rho_0})\|_{L^1_T(\dot B^{\frac{n}{p}-1}_{p,1})} \\
& \qquad 
+\|\adj(DX_v)^t\!A_v-\id\|_{L^\infty_T(\dot B^{\frac{n}{p}}_{p,1})}
\|k(\rho_0)\|_{L^\infty_T(\mathcal{M}(\dot B^{\frac{n}{p}-1}_{p,1}))}
\|\nabla(\frac{\psi}{\rho_0})\|_{L^1_T(\dot B^{\frac{n}{p}-1}_{p,1})} \\
&\lesssim 
(1+\|a_0\|_{\dot B^{\frac{n}{p}}_{p,1})})^3\|Dv\|_{L^1_T(\dot B^{\frac{n}{p}}_{p,1})}
  \|\psi \|_{L^1_T(\dot B^{\frac{n}{p}}_{p,1})}. \\
\end{array}
\end{equation*}

\noindent{\sl \underbar{Estimate of $I_6$}}~:
Owing to \eqref{eq:smallv} and to Proposition \ref{prop:flow1}, we have 
$$
\|k(J_v^{-1}\rho_0)\adj(DX_v)^t\!A_v
    \|_{L^\infty_T(\mathcal{M}(\dot B^{\frac{n}{p}-1}_{p,1}))}
\lesssim \|a_0\|_{\dot B^{\frac{n}{p}}_{p,1}}+1.
$$
Therefore, 
$$\begin{array}{lll}\|I_6(v,\psi)\|_{L^1_T(\dot B^{\frac{n}{p}-1}_{p,1})} &\!\!\!\lesssim\!\!\!& 
\|k(J^{-1}\rho_0)\adj(DX_v)^t\!A_v\|_{L^\infty_T(\mathcal{M}(\dot B^{\frac{n}{p}-1}_{p,1}))}
\||v|^2\|_{L^1_T(\dot B^{\frac{n}{p}}_{p,1})} \\[1.5ex]
&\!\!\!\lesssim\!\!\!&(\|a_0\|_{\dot B^{\frac{n}{p}}_{p,1}}+1)
\|v\|_{L^2_T(\dot B^{\frac{n}{p}}_{p,1})}^2.\end{array}$$

\noindent{\sl \underbar{Estimate of $I_7$}}~:
Recall that
$P(J_v^{-1}\rho_0,K)=\pi_0(J^{-1} \rho_0)
 + (\frac{1}{\rho_0} \psi- \frac{|v|^2}{2}) \pi_1(J^{-1} \rho_0)$. Hence
$$
I_7(v,\psi) = \adj(DX_v)
\Big( v \pi_0(J_v^{-1} \rho_0)
 + v\big(\frac{1}{\rho_0} \psi
      - \frac{|v|^2}{2}\big) \pi_1(J^{-1} \rho_0) \Big)\cdotp
$$
We already proved that if \eqref{eq:smallv} is satisfied and $1\leq p<2n$ then
\begin{equation}\nonumber
\begin{split}
&\|\adj(DX_v)\|_{L^\infty_T(\mathcal{M}(\dot B^{\frac{n}{p}-1}_{p,1}))}
\lesssim \|Dv\|_{L^1_T(\dot B^{\frac{n}{p}}_{p,1})}+1, \\
&\|\rho_0^{-1}\|_{L^\infty_T(\mathcal{M}(\dot B^{\frac{n}{p}-1}_{p,1}))}
\lesssim \|a_0\|_{\dot B^{\frac{n}{p}}_{p,1}}+1, \\ 
\text{and} \quad 
&\|\pi_1(J^{-1} \rho_0)\|_{L^\infty_T(\mathcal{M}(\dot B^{\frac{n}{p}-1}_{p,1}))}
\lesssim \|a_0\|_{\dot B^{\frac{n}{p}}_{p,1}}+1.
\end{split}
\end{equation}
Therefore, by Proposition \ref{prop:prd}, we have 
\begin{equation*}
\begin{split}
\|I_7(v,&\psi)\|_{L^1_T(\dot B^{\frac{n}{p}-1}_{p,1})} \\
&\lesssim 
\|v \pi_0(J^{-1} \rho_0)\|_{L^1_T(\dot B^{\frac{n}{p}-1}_{p,1})}
+(\|a_0\|_{\dot B^{\frac{n}{p}}_{p,1}}+1)
\|v \psi\|_{L^1_T(\dot B^{\frac{n}{p}-1}_{p,1})} \\
&\qquad\qquad
+(\|a_0\|_{\dot B^{\frac{n}{p}}_{p,1}}+1)\|v|v|^2\|_{L^1_T(\dot B^{\frac{n}{p}-1}_{p,1})} \\
&\lesssim 
\|\pi_0(J^{-1} \rho_0)\|_{L^\infty_T(\dot B^{\frac{n}{p}}_{p,1})}
T\|v\|_{L^\infty_T(\dot B^{\frac{n}{p}-1}_{p,1})} 
+(\|a_0\|_{\dot B^{\frac{n}{p}}_{p,1}}+1)
\|v\|_{L^2_T(\dot B^{\frac{n}{p}}_{p,1})} 
\|\psi\|_{L^2_T(\dot B^{\frac{n}{p}-1}_{p,1})} \\
&\qquad\qquad
+(\|a_0\|_{\dot B^{\frac{n}{p}}_{p,1}}+1)
\|v\|_{L^\infty_T(\dot B^{\frac{n}{p}-1}_{p,1})} 
\||v|^2\|_{L^1_T(\dot B^{\frac{n}{p}}_{p,1})} \\
&\lesssim 
(\|a_0\|_{\dot B^{\frac{n}{p}}_{p,1}}+1) 
\big(
T\|v\|_{L^\infty_T(\dot B^{\frac{n}{p}-1}_{p,1})} 
+\|v\|_{L^2_T(\dot B^{\frac{n}{p}}_{p,1})} 
\|\psi\|_{L^2_T(\dot B^{\frac{n}{p}-1}_{p,1})} \\
&\qquad\qquad\qquad\qquad\qquad\qquad\qquad\qquad
+\|v\|_{L^\infty_T(\dot B^{\frac{n}{p}-1}_{p,1})} 
\|v\|_{L^2_T(\dot B^{\frac{n}{p}}_{p,1})}^2
\big). \\
\end{split}
\end{equation*}

\noindent{\sl \underbar{Estimate of $I_8$}}~: Recall that
$$
I_8(v,w) = \adj(DX_v) \big(\lambda(J_{v}^{-1}\rho_0)^t\!A_v : \nabla w\ \id 
                    + \mu(J_{v}^{-1}\rho_0) (Dw \cdot A_v + {}^t\!A_v\cdot\nabla w) \big) \cdot w.
$$
{}From the previous computations, we know that for any smooth enough function $z,$
$$
\|\adj(DX_v)z(J_{v}^{-1}\rho_0)^t\!A_v
\|_{L^\infty_T(\mathcal{M}(\dot B^{\frac{n}{p}-1}_{p,1}))}
\lesssim1+\|a_0\|_{\dot B^{\frac{n}{p}}_{p,1}}.
$$
Therefore, we have by Proposition \ref{prop:prd}
$$
\|I_8(v,v)\|_{L^1_T(\dot B^{\frac{n}{p}-1}_{p,1})} 
\lesssim (1+\|a_0\|_{\dot B^{\frac{n}{p}}_{p,1}})
\|v\|_{L^\infty_T(\dot B^{\frac{n}{p}}_{p,1})}^2.
$$
In summary, we have that 
\begin{multline}\label{I1-4}
\|I_1(v,v)+I_2(v,v)+I_3(v,v)+I_4(v,\psi)
  \|_{L^1_T(\dot B^{\frac{n}{p}}_{p,1})} 
\lesssim(\|a_0\|_{\dot B^{\frac{n}{p}}_{p,1}}+1)^2\\\times
\big(T+ \|Dv\|_{L^1_T(\dot B^{\frac{n}{p}}_{p,1})}^2 
  +\|a_0\|_{\dot B^{\frac{n}{p}}_{p,1}} 
     \|v\|_{L^2_T(\dot B^{\frac{n}{p}}_{p,1})}^2
  + \|\psi\|_{L^1_T(\dot B^{\frac{n}{p}}_{p,1})} 
          \|Dv\|_{L^1_T(\dot B^{\frac{n}{p}}_{p,1})}\big)
\end{multline}
and that
\begin{multline}\label{I5-8}
\|I_5(v,\psi) + I_6(v,\psi)
      + I_7(v,\psi)+ I_8(v,v) 
  \|_{L^1_T(\dot B^{\frac{n}{p}-1}_{p,1})}\\
\lesssim
(\|a_0\|_{\dot B^{\frac{n}{p}}_{p,1}}+1)^3
\big(
T\|v\|_{L^\infty_T(\dot B^{\frac{n}{p}-1}_{p,1})} 
+ \|Dv\|_{L^1_T(\dot B^{\frac{n}{p}}_{p,1})}
  \|\psi \|_{L^1_T(\dot B^{\frac{n}{p}}_{p,1})} \\
+(\|v\|_{L^\infty_T(\dot B^{\frac{n}{p}-1}_{p,1})}+1)
     \|v\|_{L^2_T(\dot B^{\frac{n}{p}}_{p,1})}^2
 +\|v\|_{L^2_T(\dot B^{\frac{n}{p}}_{p,1})} 
   \|\psi\|_{L^2_T(\dot B^{\frac{n}{p}-1}_{p,1})} \big).
\end{multline}

Plugging inequalities \eqref{I1-4} and \eqref{I5-8} 
into \eqref{pr:apriori:LM-K}, 
we obtain
\begin{equation}
\nonumber
\begin{split}
&\|(\tilde u,\tilde K)\|_{E_p(T)} \\
&\le C e^{C_{\rho,m}T}
(\|a_0\|_{\dot B^{\frac{n}{p}}_{p,1}}+1)^3
\Big(\|Du_L\|_{L^1_T(\dot B^{\frac{n}{p}}_{p,1})} 
 + \|K_L\|_{L^1_T(\dot B^{\frac{n}{p}}_{p,1})} \\
&\qquad\qquad\qquad\quad
+ T(\|v\|_{L^\infty_T(\dot B^{\frac{n}{p}-1}_{p,1})}+1)
+ \|Dv\|_{L^1_T(\dot B^{\frac{n}{p}}_{p,1})}^2 
+ \|Dv\|_{L^1_T(\dot B^{\frac{n}{p}}_{p,1})}
   \|\psi \|_{L^1_T(\dot B^{\frac{n}{p}}_{p,1})} \\
&\qquad\qquad\qquad\quad
  +(\|v\|_{L^\infty_T(\dot B^{\frac{n}{p}-1}_{p,1})}+1)
   \|v\|_{L^2_T(\dot B^{\frac{n}{p}}_{p,1})}^2
  +\|v\|_{L^2_T(\dot B^{\frac{n}{p}}_{p,1})} 
\|\psi\|_{L^2_T(\dot B^{\frac{n}{p}-1}_{p,1})} \Big). \\
\end{split}
\end{equation}

Note that by the linear parabolic estimate, we have 
$\|u_L\|_{L^\infty_T(\dot B^{\frac{n}{p}-1}_{p,1})}
\le \|u_0\|_{\dot B^{\frac{n}{p}-1}_{p,1}}$. 
Since $(v,\psi)$ belongs to the ball 
$\bar B_{E_p(T)}((u_L, K_L),R)$,
decomposing $v$ into $\tilde v + u_L$ 
and $\psi$ into $\tilde \psi + K_L$
gives us
\begin{equation}
\nonumber
\begin{split}
&\|(\tilde u,\tilde K)\|_{E_p(T)} 
\le C e^{C_{\rho,m}T}
(\|a_0\|_{\dot B^{\frac{n}{p}}_{p,1}}+1)^3
\Big(\|Du_L\|_{L^1_T(\dot B^{\frac{n}{p}}_{p,1})} 
 + \|K_L\|_{L^1_T(\dot B^{\frac{n}{p}}_{p,1})} \\
&\qquad\qquad\qquad\quad
+ T(\|u_0\|_{\dot B^{\frac{n}{p}-1}_{p,1}}+R+1)
+ \|Du_L\|_{L^1_T(\dot B^{\frac{n}{p}}_{p,1})}^2+R^2 \\
&\qquad\qquad\qquad\quad
+ (\|Du_L\|_{L^1_T(\dot B^{\frac{n}{p}}_{p,1})}+R)
  (\|K_L\|_{L^1_T(\dot B^{\frac{n}{p}}_{p,1})}+R) \\
&\qquad\qquad\qquad\quad
  +(\|u_0\|_{\dot B^{\frac{n}{p}-1}_{p,1}} +R+1)
   \|u_L\|_{L^2_T(\dot B^{\frac{n}{p}}_{p,1})}^2+R^2 \\
&\qquad\qquad\qquad\quad
  +(\|u_L\|_{L^2_T(\dot B^{\frac{n}{p}}_{p,1})}+R) 
   (\|K_L\|_{L^2_T(\dot B^{\frac{n}{p}-1}_{p,1})}+R) \Big) \\
&\le C e^{C_{\rho,m}T}
(\|a_0\|_{\dot B^{\frac{n}{p}}_{p,1}}+1)^3
(\|u_0\|_{\dot B^{\frac{n}{p}-1}_{p,1}}+1+R) \Big(\|Du_L\|_{L^1_T(\dot B^{\frac{n}{p}}_{p,1})} 
 + \|K_L\|_{L^1_T(\dot B^{\frac{n}{p}}_{p,1})}\\& 
\qquad\qquad+ T + \|Du_L\|_{L^1_T(\dot B^{\frac{n}{p}}_{p,1})}^2+ \|K_L\|_{L^1_T(\dot B^{\frac{n}{p}}_{p,1})}^2 
  +\|u_L\|_{L^2_T(\dot B^{\frac{n}{p}}_{p,1})}^2 +\|K_L\|_{L^2_T(\dot B^{\frac{n}{p}-1}_{p,1})}^2+R^2\Big). 
\end{split}
\end{equation}

We first choose $R$ so that for a small enough 
constant $\eta$, 
\begin{equation}
\label{eta}
2C(\|a_0\|_{\dot B^{\frac{n}{p}}_{p,1}}+1)^3
(\|u_0\|_{\dot B^{\frac{n}{p}-1}_{p,1}}+1) R\le \eta
\end{equation}
and take $T$ so that 
\begin{eqnarray}\label{TR1}
& C_{\rho_0,m} T \le \log 2,
\quad T \le R^2,
\quad \|Du_L\|_{L^1_T(\dot B^{\frac{n}{p}}_{p,1})} \le R, \\
&\quad \|K_L\|_{L^1_T(\dot B^{\frac{n}{p}}_{p,1})} \le R, 
\quad \|u_L\|_{L^2_T(\dot B^{\frac{n}{p}}_{p,1})} \le R, 
\quad \|K_L\|_{L^2_T(\dot B^{\frac{n}{p}-1}_{p,1})}\le R,
\end{eqnarray}
then we may conclude that $\Phi$ is a self-map on 
the ball $\bar B_{E_p(T)}((u_L,K_L),R)$.


\medbreak\noindent{\sl Second step : Contraction estimate.}
\smallbreak
We set $u_j:=\Phi_1(v_j, \psi_j)$, 
$K_j:=\Phi_2(v_j, \psi_j)$ for $j=1,2,$ 
and $\du := u_2 - u_1$ and $\dK := K_2 - K_1$. To simplify the notation, we set $X_i:=X_{v_i},$ $A_i:=A_{v_i}$ and $J_i:=J_{v_i}.$

In order to prove that $\Phi$ is contractive, 
it is just a matter of applying Proposition \ref{apriori:LM-K} to the system fulfilled by $(\du,\dK),$ namely
$$
\left\{
\begin{array}{l}
L_{\rho_0} \du
+ \rho_0^{-1}\nabla( \rho_0^{-1} \pi_1(\rho_0) \dK) \\[1ex]
\qquad = \rho_0^{-1}\div\big( \sum_{j=1}^3(I_j(v_2,v_2)-I_j(v_1,v_1)) 
     +(I_4(v_2,\psi_2)-I_4(v_1,\psi_1)) \big), \\[2ex]
H_{\rho_0} \dK = \div \big(\sum_{j=1}^3 (I_j(v_2,\psi_2)- I_j(v_1, \psi_1))
  + (I_8(v_2,v_2)- I_8(v_1,v_1)) \big).
\end{array}
\right.
$$ 
Taking advantage of the computations in \cite{danchin2}, we get for $j=1,2,3,$
$$
\|I_j(v_2,v_2)-I_j(v_1,v_1) \|_{L^1_T(\dot B^{\frac{n}{p}}_{p,1})}
\leq C_{\rho_0} \|(Dv_1,Dv_2)\|_{L^1_T(\dot B^{\frac{n}{p}}_{p,1})}
\|D\dv\|_{L^1_T(\dot B^{\frac{n}{p}}_{p,1})}. 
$$
Concerning the pressure term, a straightforward 
calculation based on Proposition \ref{prop:flow2} ensures that for some constant $C_{\rho_0}$ depending only on $\rho_0,$ $n$ and $p,$
$$
\|I_4(v_2,\psi_2)-I_4(v_1,\psi_1) \|_{L^1_T(\dot B^{\frac{n}{p}}_{p,1})}
\leq C_{\rho_0} \|(Dv_1,Dv_2)\|_{L^1_T(\dot B^{\frac{n}{p}}_{p,1})}
\|(\dv,\dpsi)\|_{E_p(T)}. $$
Indeed: 
$$\displaylines{
I_4(v,\psi) = -\adj(DX_v) \pi_0(J_v^{-1}\rho_0) 
-(\adj(DX_v)\frac{\pi_1(J_v^{-1}\rho_0)}{\rho_0}
                -\frac{\pi_1(\rho_0)}{\rho_0}) \psi \id\hfill\cr\hfill
 +\adj(DX_v) \pi_1(J_v^{-1}\rho_0)\frac{|v|^2}{2}\cdotp}
 $$
 Hence
$$\displaylines{
I_4(v_2,\psi_2)-I_4(v_1,\psi_1) =- \big( \adj(DX_{2}) \pi_0(J_{2}^{-1}\rho_0) 
          - \adj(DX_{1}) \pi_0(J_{1}^{-1}\rho_0) \big) \hfill\cr\hfill
- \big( \adj(DX_{2})\frac{\pi_1(J_{2}^{-1}\rho_0)}{\rho_0}\psi_2 
         - \adj(DX_{1})\frac{\pi_1(J_{1}^{-1}\rho_0)}{\rho_0} \psi_1 \big)
 + \frac{\pi_1(\rho_0)}{\rho_0} \dpsi \, \id\,  \hfill\cr\hfill
 +\frac12\Bigl( \adj(DX_{2}) \pi_1(J_{2}^{-1}\rho_0)
  - \adj(DX_{1}) \pi_1(J_{1}^{-1}\rho_0)\Bigr)|v_1|^2
  +\frac12\adj(DX_2)\pi_1(J_2^{-1}\rho_0)\dv\cdot(v_2+v_1). }
$$
Now we have, for the first term of the above equality, 
\begin{equation}
\nonumber
\begin{split}
\|\adj(DX_{2}) &\pi_0(J_{2}^{-1}\rho_0) 
    - \adj(DX_{1}) \pi_0(J_{1}^{-1}\rho_0) 
  \|_{L^1_T(\dot B^{\frac{n}{p}}_{p,1})} \\
&\lesssim
\|\adj(DX_{2}) (\pi_0(J_{2}^{-1}\rho_0) - \pi_0(J_{1}^{-1}\rho_0) )
\|_{L^1_T(\dot B^{\frac{n}{p}}_{p,1})} \\
&\qquad\qquad
+\|(\adj(DX_{2}) - \adj(DX_{1}) ) \pi_0(J_{1}^{-1}\rho_0) 
 \|_{L^1_T(\dot B^{\frac{n}{p}}_{p,1})} \\
&\leq C_{\rho_0} T \|(Dv_1,Dv_2)\|_{L^1_T(\dot B^{\frac{n}{p}}_{p,1})}
\|D\dv\|_{L^1_T(\dot B^{\frac{n}{p}}_{p,1})}. 
\end{split}
\end{equation}
For the second and third terms, it is easily obtained that 
\begin{equation}
\nonumber
\begin{split}
\|\big( \adj(DX_{2})&\frac{\pi_1(J_{2}^{-1}\rho_0)}{\rho_0}\psi_2 
         - \adj(DX_{1})\frac{\pi_1(J_{1}^{-1}\rho_0)}{\rho_0} \psi_1 \big)
 - \frac{\pi_1(\rho_0)}{\rho_0} \,\dpsi\, \id
  \|_{L^1_T(\dot B^{\frac{n}{p}}_{p,1})} \\
&\lesssim
\|\big(\adj(DX_{2}) - \adj(DX_{1}) \big)
  \frac{\pi_1(J_{2}^{-1}\rho_0)}{\rho_0}\psi_2 
\|_{L^1_T(\dot B^{\frac{n}{p}}_{p,1})} \\
&\quad
+\|\adj(DX_{1})\big( \frac{\pi_1(J_{2}^{-1}\rho_0)}{\rho_0}
      -\frac{\pi_1(J_{1}^{-1}\rho_0)}{\rho_0} \big)\psi_2 
 \|_{L^1_T(\dot B^{\frac{n}{p}}_{p,1})} \\
&\quad
+\|\Big(\adj(DX_{1}) \frac{\pi_1(J_{1}^{-1}\rho_0)}{\rho_0} 
   -\frac{\pi_1(\rho_0)}{\rho_0} \id \Big)\dpsi 
 \|_{L^1_T(\dot B^{\frac{n}{p}}_{p,1})} \\
&\leq C_{\rho_0} \|(Dv_1,Dv_2,\psi_2)\|_{L^1_T(\dot B^{\frac{n}{p}}_{p,1})}
\|(D\dv, \dpsi)\|_{L^1_T(\dot B^{\frac{n}{p}}_{p,1})}. 
\end{split}
\end{equation}
For the last terms, in the same manner as above, we may check that 
$$
\displaylines{
\!\|( \adj(DX_{2}) \pi_1(J_{2}^{-1}\rho_0)- \adj(DX_{1}) \pi_1(J_{1}^{-1}\rho_0))|v_1|^2\|_{L_T^1(\dot B^{\frac np}_{p,1})}
\!\leq C_{\!\rho_0} \|D\dv\|_{L^1_T(\dot B^{\frac{n}{p}}_{p,1})}
\|v_1\|_{L_T^2(\dot B^{\frac np}_{p,1})}^2\cr
\|\adj(DX_{2}) \pi_1(J_{2}^{-1}\rho_0)\: \dv\cdot (v_2+v_1)\|_{L_T^1(\dot B^{\frac np}_{p,1})}
\leq C_{\rho_0} \|\dv\|_{L^2_T(\dot B^{\frac{n}{p}}_{p,1})}
\|v_1+v_2\|_{L_T^2(\dot B^{\frac np}_{p,1})}.}
$$
Finally, to handle the terms $I_5,$ $I_6,$ $I_7$ and $I_8,$ we use the following decompositions:
\begin{equation}
\nonumber
\begin{split}
I_5(v_2,\psi_2) - I_5(v_1,\psi_1) 
&= \big( k(J_{2}^{-1}\rho_0)-k(J_{1}^{-1}\rho_0) \big) 
   \adj(DX_{2})^t\!A_{2}\nabla(\frac{\psi_2}{\rho_0}) \\
&\qquad
 + k(J_{v_1}^{-1}\rho_0) ( \adj(DX_{2})-\adj(DX_{1}) )
   ^t\!A_{2}\nabla(\frac{\psi_2}{\rho_0}) \\
&\qquad
 + k(J_{1}^{-1}\rho_0)\adj(DX_{1})
 ^t\!(A_{2}-A_{1})\nabla(\frac{\psi_2}{\rho_0}) \\
&\qquad
 + \big(k(J_{v_1}^{-1}\rho_0)\adj(DX_{1})^t\!A_{1} - \id \big)
    \nabla(\frac{\dpsi}{\rho_0}),
\end{split}
\end{equation}
\begin{equation}
\nonumber
\begin{split}
I_6(v_2,\psi_2) - I_6(v_1,\psi_1) 
&= (k(J_{2}^{-1}\rho_0) - k(J_{1}^{-1}\rho_0) ) 
    \adj(DX_{2}){}^t\!A_{2} \nabla (\frac{|v_2|^2}{2})\\
&\qquad  + k(J_{1}^{-1}\rho_0)(\adj(DX_{2})-\adj(DX_{1}))
    {}^t\!A_{2} \nabla (\frac{|v_2|^2}{2})\\
&\qquad
  + k(J_{1}^{-1}\rho_0)\adj(DX_{1})
   {}^t\!(A_{2} -A_{1})\nabla (\frac{|v_2|^2}{2})\\
&\qquad
  + k(J_{1}^{-1}\rho_0)\adj(DX_{1}){}^t\!A_{1} 
    \nabla (\frac{|v_2|^2}{2} - \frac{|v_1|^2}{2}),
\end{split}
\end{equation}

$$\displaylines{
I_7(v_2,\psi_2) - I_7(v_1,\psi_1) = \big( \adj(DX_{2}) v_2   - \adj(DX_{1}) v_1 \big)
   \big( \pi_0(J_{2}^{-1} \rho_0)
         + \frac{\psi_2}{\rho_0}-\frac{|v_2|^2}{2}) \pi_1(J_{2}^{-1}\rho_0) \big) \hfill\cr\hfill
  + \adj(DX_{1}) v_1 
    \big( \pi_0(J_{2}^{-1} \rho_0) - \pi_0(J_{1}^{-1} \rho_0) \big) \hfill\cr\hfill 
  + \adj(DX_{1}) v_1      \big( (\frac{\psi_2}{\rho_0}-\frac{|v_2|^2}{2}) \pi_1(J_{2}^{-1}\rho_0) 
         - (\frac{\psi)1}{\rho_0} - \frac{|v_1|^2}{2}) \pi_1(J_{1}^{-1} \rho_0) \big),}
             $$
$$
\displaylines{
I_8(v_2,\psi_2) - I_8(v_1,\psi_1) 
= \adj(DX_{2}) 
   \big(\lambda(J_{2}^{-1}\rho_0)\div_{A_{2}} v_2 \ \id 
          + 2\mu(J_{2}^{-1}\rho_0) D_{A_{2}}(v_2) \big) \cdot v_2 \hfill\cr\hfill
 -\adj(DX_{1})    \big(\lambda(J_{1}^{-1}\rho_0)\div_{A_{1}} v_1 \ \id 
          + 2\mu(J_{1}^{-1}\rho_0) D_{A_{1}}(v_1) \big) \cdot v_1.}
$$
Then using Proposition \ref{prop:flow2},  it is easy to see that for $j=5,6,7,8,$ we have
\begin{equation}
\|I_j(v_2,v_2)-I_j(v_1,v_1) \|_{L^1_T(\dot B^{\frac{n}{p}}_{p,1})}\leq
C_{\rho_0} 
\|(Dv_1,Dv_2,\psi_1,\psi_2)\|_{E_p(T)}
\|(\dv,\dpsi)\|_{E_p(T)}.
\end{equation}
Proposition \ref{apriori:LM-K} gives us that 
\begin{equation*}
\begin{split}
&\|(\du,\dK)\|_{E_p(T)} \\
&\le Ce^{C_{\rho,m}T} 
\Big(\Big\|\sum_{j=1}^3(I_j(v_2,v_2)-I_j(v_1,v_1)) 
         +(I_4(v_2,\psi_2)-I_4(v_1,\psi_1))
  \Big\|_{L^1_T(\dot B^{\frac{n}{p}}_{p,1})} \\
&\qquad\qquad\qquad
+ \Big\|\sum_{j=5}^7 (I_j(v_2,\psi_2)- I_j(v_1, \psi_1))
         + (I_8(v_2,v_2)- I_8(v_1,v_1))
  \Big\|_{L^1_T(\dot B^{\frac{n}{p}-1}_{p,1})} 
\Big)\\
&\le  Ce^{C_{\rho,m}T} \|(v_1,v_2,\psi_1,\psi_2)\|_{E_p(T)}
\|(\dv,\dpsi)\|_{E_p(T)}.
\end{split}
\end{equation*}

Given that $v_j, \psi_j \in \bar B_{E_p(T)}((u_L,K_L), R)$
($j=1,2$), taking  $\eta$, $T$ and $R$ smaller as the case may be, we end up with 
\begin{equation}
\begin{split}
\|(\du,\dK)\|_{E_p(T)}
&\le \frac{1}{2}
\|(\dv,\dpsi)|_{E_p(T)}. \\
\end{split}
\end{equation}
One can thus conclude that $\Phi$ admits a unique fixed point in 
$\bar B_{E_p(T)}\bigl((u_L,K_L), R\bigr)$.

\medbreak\noindent{\sl Third step: Regularity of the density.}
\smallbreak
Granted with the above velocity field
$u$ in $E_p(T)$, we set $\rho:=J_u^{-1}\rho_0$.
By construction, the triplet $(\rho,u,K)$ satisfies 
\eqref{CNS-l}. In order to prove that $a:=\rho-1$ 
is in $\cC([0,T];\dot B^{\frac{n}{p}}_{p,1}),$ we use the fact that 
\begin{equation*}
a=(J_u^{-1}-1)a_0+a_0.
\end{equation*}
Given Proposition \ref{prop:flow1}, and the fact that $Du\in L^1(0,T;\dot B^{\frac{n}{p}}_{p,1})$,
it is clear that 
$(J_u^{-1}-1)$ belongs to 
$\cC([0,T];\dot B^{\frac{n}{p}}_{p,1})$.
Hence $a$ belongs to $\cC([0,T];\dot B^{\frac{n}{p}}_{p,1})$, too. 
Because $\dot B^{\frac{n}{p}}_{p,1}$ is continuously 
embedded in $L^\infty$, 
Condition $\inf_{x} \rho>0$ is fulfilled 
on $[0,T]$ (taking smaller $T$ if needed).

\medbreak\noindent{\sl Last step: Uniqueness and continuity of the flow map.}
\smallbreak
We now consider two triplets
$(\rho_{01}, u_{01}, K_{01})$ and $(\rho_{02}, u_{02}, K_{02})$ 
of data fulfilling the assumptions of 
Theorem \ref{thm1} and we 
denote by $(\rho_1,u_1, K_1)$ and $(\rho_2,u_2, K_2)$ 
two solutions with $(a_1,u_1, K_1)$ and 
$(a_2,u_2, K_2)$ in $E_p(T)$ corresponding 
to those data. 
Let 
$M_{\rho_{0j}} K 
    := \rho_{0j}^{-1} \nabla (\rho_{0j}^{-1} \pi_1(\rho_{0j}^{-1}) K)$. 
Making difference of the two equations corresponding to 
$(\rho_1,u_1)$ and $(\rho_2,u_2)$, we have
\begin{equation*}
\begin{split}
&L_{\rho_{02}} u_2-L_{\rho_{01}}u_1 
= L_{\rho_{01}} \du + (L_{\rho_{02}} -L_{\rho_{01}}) u_2 \\
\text{and} \quad 
&K_{\rho_{02}} K_2-K_{\rho_{01}}K_1 
= M_{\rho_{01}} \dK + (M_{\rho_{02}} -M_{\rho_{01}}) K_2. \\
\end{split}
\end{equation*}
Setting $\du := u_2 - u_1$
and $\dK := K_2 - K_1$, we thus get  
\begin{equation}\label{diff}
\left\{
\begin{array}{l}
L_{\rho_{01}} \du + M_{\rho_{01}} \dK
 = (L_{\rho_{01}}-L_{\rho_{02}}) (u_2) + (M_{\rho_{01}} -M_{\rho_{02}}) K_2\\
 \quad+ (\rho_{02})^{-1} \div \Big( \Sum_{j=1}^3 (I_j^2(u_2,u_2)-I_j^2(u_1,u_1))
     +(I_4^2(u_2,\psi_2)-I_4^2(u_1,\psi_1)) \Big)  \\
 \quad+ \big((\rho_{02})^{-1}-(\rho_{01})^{-1}\big) 
 \div \Big( \Sum_{j=1}^3 I_j^2(u_1,u_1)+I_4^2(u_1,\psi_1) \Big)\\
 \quad+ (\rho_{01})^{-1} \div \Big( \Sum_{j=1}^3 (I^2_j-I^1_j)(u_1,u_1)
     +(I_4^2-I_4^1)(u_1,\psi_1)\Big), \\[2ex]
H_{\rho_{01}} \dK
 = \div \big(\Sum_{j=1}^3 (I_j^2(u_2,K_2)- I_j^2(u_1, K_1))
  + (I_8^2(u_2,u_2)- I_8^2(u_1,u_1)) \big) \\
 \qquad\qquad\qquad\qquad+ \div \Big(\Sum_{j=1}^3 (I_j^2- I_j^1)(u_1, K_1)
  + (I_8^2- I_8^1)(u_1,u_1) \Big), 
\end{array}
\right.
\end{equation} 
where $I^i_j$ ($j=1,...,5$) 
correspond to the quantities that have been defined 
previously in \eqref{terms} with density $\rho_{0i}$ for $i=1,2$:
\begin{equation*}
\begin{split}
&I_1^i(v,w) := (\adj(DX_v)- Id)\big(\mu(J_v^{-1}\rho_{0i})(DwA_v +{}^t\!A_v \nabla w) \\
&\qquad\qquad\qquad 
     + \lambda(J_v^{-1}\rho_{0i})(^t\!A_v:\nabla w)\id\big), \\
&I_2^i(v,w) := (\mu(J_v^{-1}\rho_{0i})-\mu(\rho_0^i))(Dw\cdot A_v+{}^t\!A_v\cdot \nabla w) \\
&\qquad\qquad\qquad 
     + (\lambda(J_v^{-1}\rho_{0i})-\lambda(\rho_{0i}))(^t\!A_v:\nabla w)\id, \\
&I_3^i(v,w) := \mu(\rho_{0i})(Dw(A_v-\id)+^t\!(A_v-\id) \nabla w)
              + \lambda(\rho_{0i})(^t\!(A_v-\id):\nabla w)\id, \\
&I_4^i(v,\psi) := -\adj(DX_v)P(J_v^{-1}\rho_{0i},\psi), \\
&I_5^i(v,\psi) := ( k(J_v^{-1}\rho_{0i})\adj(DX_v)^t\!A_v 
                 - k(\rho_{0i})\id) \nabla(\frac{\psi}{\rho_{0i}}), \\
&I_6^i(v,\psi) := k(J_v^{-1}\rho_{0i})\adj(DX_v)^t\!A_v \nabla (\frac{|v|^2}{2}), \\
&I_7^i(v,\psi) := \adj(DX_v) v P(J_v^{-1}\rho_{0i}, \psi), \\
\text{and} \ 
&I_8^i(v,w) := \adj(DX_v) \big(\lambda(J_v^{-1}\rho_{0i})\div_{A_v} w \ \id 
                    + 2\mu(J_v^{-1}\rho_{0i}) D_{A_v}(w) \big) \cdot w. \\ 
\end{split}
\end{equation*}
The proof is carried out by applying Proposition 
\ref{apriori:LM-K} to \eqref{diff} and 
using Proposition \ref{prop:flow2} to estimate 
each term on the left-hand side of \eqref{diff}, 
exactly as in the second step. 
\medbreak
Assuming that $\da_0$, $\du_0$ 
and $\dK_0$ are small enough,  a bootstrap argument will provide us with, for small enough $t,$
\begin{equation*}
\begin{split}
\|(\du,\dK)\|_{E_p(t)}
&\le C_{\rho_{01}, \rho_{02}}
\big( \|\da_0\|_{\dot B^{\frac{n}{p}}_{p,1}} 
+ \|\du_0\|_{\dot B^{\frac{n}{p}-1}_{p,1}}
+ \|\dK_0\|_{\dot B^{\frac{n}{p}-2}_{p,1}}\big). \\
\end{split}
\end{equation*}
Regarding the density, we have
\begin{equation*}
\begin{split}
\da = J_{u_1}^{-1}\da_0 + (J_{u_2}^{-1}-J_{u_1}^{-1})a_{02}.
\end{split}
\end{equation*}
Hence for all $t\in[0,T]$,
\begin{equation}
\nonumber
\label{cont-flow}
\begin{split}
\|\da(t)\|_{\dot B^{\frac{n}{p}}_{p,1}}
&\le C (1+\|D\du\|_{L^1_t(\dot B^{\frac{n}{p}}_{p,1})})
\big( \|\du_0\|_{\dot B^{\frac{n}{p}}_{p,1}}
+ \|\da_0\|_{\dot B^{\frac{n}{p}}_{p,1}} \big). \\
\end{split}
\end{equation}
Therefore, we may conclude to both  uniqueness and 
continuity of the data-solution map on 
a small enough time interval. 
Iterating the proof will yield 
uniqueness on the whole time interval $[0,T]$.


\subsection{Proof of Theorem \ref{thm2}}
Finally, we consider the possibility of 
reverting back the solution obtained in the Lagrangian 
coordinates to that in the Eulerian coordinates.
Theorem \ref{thm2} is a corollary of the 
following proposition which 
states that, under the restriction $1<p<n$ and $n\ge3$, 
Systems \eqref{CNS} and \eqref{CNS-l} 
(and consequently \eqref{CNSE} and \eqref{CNSE-l} as well)
are equivalent in our functional framework. 
\begin{prop}
\label{prop-equiv}
Let $1<p<n$ with $n\ge3,$ and  
$(\rho,u,\theta)$ be a solution to \eqref{CNS} with 
$\rho-1\in \cC([0,T];\dot B^{\frac{n}{p}}_{p,1})$, 
$(u,\theta)\in E_p(T)$ and, for small enough $c,$
\begin{equation}\label{cond-u}
\int_0^T \|\nabla u\|_{\dot B^{\frac{n}{p}}_{p,1}}\le c.
\end{equation}
Let $X$ be the flow of $u$ defined in \eqref{def-flow}
and $E,$ the {\sl total energy by unit volume} defined 
in \eqref{def-E}. 
Then after defining the triplet 
$(\overline\rho, \overline u, \overline E)
 :=(\rho\circ X, u\circ X, E\circ X)$ 
and 
$\overline K$ as in \eqref{def-K}, 
the triplet $(\overline\rho, \overline u, \overline K)$ 
belongs to the same functional space as $(\rho,u,\theta)$ 
and satisfies \eqref{CNS-l}.

Conversely, if $\overline\rho-1 \in \cC([0,T];\dot B^{\frac{n}{p}}_{p,1})$, 
$(\overline u, \overline K) \in E_p(T)$,  
and $(\overline\rho,\overline u, \overline K)$ satisfies 
\eqref{CNS-l} and, 
for a small enough constant $c$, 
\begin{equation}
\label{cond-baru}
\int_0^T \|\nabla \overline u\|_{\dot B^{\frac{n}{p}}_{p,1}}
 \le c
\end{equation}
then the map $X$ defined in \eqref{def-flow2}
 is a $\cC^1$ 
(and in fact locally $\dot B^{\frac{n}{p}+1}_{p,1}$) 
diffeomorphism over $\R^n$ 
and after having defined $\overline E := J^{-1} \overline K$, 
$(\rho,u,E)
:=(\overline\rho\circ X^{-1},
   \overline u\circ X^{-1}, \overline E\circ X^{-1})$ 
and 
$\theta := \frac{E}{\rho}-\frac{|u|^2}{2}$, 
the triplet $(\rho,u,\theta)$  has the same regularity as 
$(\overline\rho,\overline u, \overline K)$ and satisfies \eqref{CNS}.

\end{prop}
\begin{proof} For a solution $(\rho,u,\theta)$ to \eqref{CNS} 
with the above properties, the definition of $X$ 
in \eqref{def-flow} implies that 
$DX-\id \in \cC([0,T];\dot B^{\frac{n}{p}}_{p,1})$. 
In addition, having defined $E$ as in \eqref{def-E}, 
Proposition \ref{flow-regr} ensures 
that $(\overline\rho,\overline u, \overline E)$ 
lies in the same functional space as $(\rho,u,E)$, 
and Proposition \ref{prop:flow1} ensures that 
$A-\id$, $\adj(DX)-\id$ and $J^{\pm1}-1$ 
are in $\cC([0,T];\dot B^{\frac{n}{p}}_{p,1})$. 
After performing the change of variable, 
let us define $\overline K:=J\overline E$; 
then it is clear by 
$J^{-1}-1\in \cC([0,T];\dot B^{\frac{n}{p}}_{p,1})$ 
and the product laws 
that $K$ also lies in the same space as $E$ 
\emph{provided that $1<p<n$ and $n\geq3$} (see Proposition \ref{prop:prd}).
So eventually, under this latter condition, 
$(\overline\rho,\overline u, \overline K)$ fulfills \eqref{CNS-l}
and belongs to $(1+\cC([0,T];\dot B^{\frac np}_{p,1}))\times E_p(T).$

\medbreak
Conversely, let us assume 
that we are given some solution 
$(\overline\rho,\overline u, \overline K)$  
to \eqref{CNS-l} with 
$$\overline\rho \in\cC([0,T];\dot B^{\frac{n}{p}}_{p,1})\ \hbox{ and }\  
(\overline u, \overline K) \in E_p(T).$$ 
Then one may prove that, under Condition \eqref{cond-baru}, 
the ``flow" $X(t,\dot)$ of $(\overline u, \overline K)$
defined by \eqref{def-flow2}
is a $\cC^1$-diffeomorphism over $\R^n$ 
(see \cite{danchin2}  and \cite{danchin-mucha}), 
and satisfies 
$DX-\id \in \cC([0,T];\dot B^{\frac{n}{p}}_{p,1})$. 
We follow the above steps from backward: 
first we define $\overline E:=J^{-1}\overline K$, 
then clearly $\overline E$ satisfies 
\eqref{CNSE-l} under $1<p<n$ and $n\ge3$. 
Now, one may perform the change of variables 
\begin{equation}\nonumber
(\rho,u,E):=(\overline\rho\circ X^{-1}, 
   \overline u\circ X^{-1}, 
   \overline E\circ X^{-1})
\end{equation}
and set $\theta := \frac{E}{\rho}-\frac{|u|^2}{2}$ 
to confirm that $(\rho,u,\theta)$ is indeed 
a solution to \eqref{CNSE}. 
Proposition \ref{flow-regr} ensures 
that $(\rho,u,\theta)$ has the desired regularity. 
\end{proof}

\begin{proof}[Proof of Theorem \ref{thm2}]
We consider data $(\rho_0, u_0, \theta_0)$ with 
$\rho_0$ bounded away from 0, 
$(\rho_0-1)\in\dot B^{\frac{n}{p}}_{p,1}$, 
$u_0\in\dot B^{\frac{n}{p}-1}_{p,1}$
and $\theta_0\in\dot B^{\frac{n}{p}-2}_{p,1}$. Defining $K_0$ according to \eqref{def-K0} and
observing that $n\geq3$ and $1\leq p<n$ implies $K_0\in\dot B^{\frac np-2}_{p,1},$ 
Then Theorem \ref{thm1} provides a local 
solution $(\overline \rho, \overline u, \overline K)$ to System \eqref{CNS-l} with 
$\overline \rho\in\cC([0,T];\dot B^{\frac{n}{p}}_{p,1})$ and
$(\overline u, \overline K)\in E_p(T)$. 
If $T$ is small enough then \eqref{cond-u} is
satisfied so Proposition \ref{prop-equiv} ensures that 
\begin{equation}
\nonumber
(\rho,u,\theta)
:=(\overline\rho\circ X^{-1},
   \overline u\circ X^{-1}, 
   \frac{\overline E\circ X^{-1}}{\overline\rho}-\frac{|\overline u\circ X^{-1}|^2}{2} )
\end{equation}
is a solution of \eqref{CNS} in the desired functional space. 

To prove uniqueness, we consider two 
Eulerian solutions 
$(\rho_1,u_1,\theta_1)$ and $(\rho_2,u_2,\theta_2)$ 
corresponding to the same data $(\rho_0,u_0,\theta_0)$. 
We then rewrite the system in the form of \eqref{CNSE} as before 
and perform the Lagrangian change of variables 
(pertaining to the flow of $u_1$ and $u_2$ respectively). 
The obtained triplets $(\overline\rho_1,\overline u_1,\overline K_1)$ and 
$(\overline\rho_2,\overline u_2,\overline K_2)$ 
(where $\overline K_j := J_{u_j} \overline E_j$ with 
$E_j$ as before) 
are in $(1+\cC([0,T];\dot B^{\frac np}_{p,1}))\times E_p(T),$ and both satisfy  
\eqref{CNS-l} with the same $(\rho_0,u_0,K_0)$ 
(with $K_0$ defined as in \eqref{def-K0}). 
Hence they coincide, as a consequence of the 
uniqueness part of Theorem \ref{thm1}.
\end{proof}

\medbreak
We conclude this section with a short discussion about  the cases $n=2$, or  $n\geq3$ and  $n\leq p<2n.$

As already pointed out in the introduction, owing to the product laws (see Proposition \ref{prop:prd}), it is no longer possible 
to deduce that $K_0$ (or $E_0$) is in $\dot B^{\frac np-2}_{p,1}$ from the hypothesis that $a_0\in\dot B^{\frac np}_{p,1},$
$u_0\in\dot B^{\frac np-1}_{p,1}$ and $\theta_0\in\dot B^{\frac np-2}_{p,1}.$ 
Therefore it is suitable to look at the equivalence between the Lagrangian Navier-Stokes equations \eqref{CNS-l},
and the Eulerian Navier-Stokes equations written in terms of $(\rho,u,E)$ (namely \eqref{CNSE}), rather 
than in terms of $(\rho,u,\theta).$ 
In this new setting, one can mimic the proofs of Proposition \ref{prop-equiv} and Theorem \ref{thm2}.
The only difference concerns the regularity issue  when making the change 
from $\overline K$ to $\overline E :=J^{-1}\overline K$ (or conversely).
Indeed, from $J^{\pm1}-1\in \cC([0,T];\dot B^{\frac np}_{p,1})$ and $\overline K\in\cC([0,T];\dot B^{\frac np-2}_{p,1})$
it is no longer possible to deduce that $\overline E$ is in $\cC([0,T];\dot B^{\frac np-2}_{p,1}),$ because 
Condition $n/p-2>-\min(n/p,n/p')$   in Proposition \ref{prop:prd} is not fulfilled. 
At the same time, arguing by interpolation, we see that  the solution $(\overline a,\overline u,\overline K)$ constructed
in Theorem \ref{thm1} is such that
$$
\overline K\in L_T^{\frac1{1-\delta}}(\dot B^{\frac np-2\delta}_{p,1})\quad\hbox{for all }\ \delta\in [0,1].
$$
As  $J^{\pm1}-1\in \cC([0,T];\dot B^{\frac np}_{p,1}),$ we conclude that  
$\overline E\in  L_T^{\frac1{1-\delta}}(\dot B^{\frac np-2\delta}_{p,1})$ whenever $n\geq2,$ $\delta<1$ and $p<n\min(2,1/\delta).$ 
Then it is easy to conclude to the following corollary:
\begin{cor}\label{cor1}
Under the  assumptions of Theorem \ref{thm1} with $n=2$ and $1<p<4,$ or $n\geq3$ and $n\leq p<2n,$  
System \eqref{CNSE} has a unique   local solution $(\rho,u,E)$ with $\rho-1\in\cC([0,T];\dot B^{\frac np-1}_{p,1}),$ 
$u\in\cC([0,T];\dot B^{\frac np-1}_{p,1})\cap L^1(0,T;\dot B^{\frac np+1}_{p,1})$ and   
$E \in  L^{\frac1{1-\delta}}(0,T;\dot B^{\frac np-2\delta}_{p,1})$ for all $\delta\in[0,n/p).$ 
 \end{cor}


\section{Appendix}

This section is devoted to presenting some technical results that have been used repeatedly 
in the paper. In the first paragraph, we recall basic nonlinear estimates 
involving Besov norms.  Next, we state estimates for the flow. Finally, 
we give some details on how \eqref{CNS-l} may be derived from \eqref{CNSE}.

\subsection{Estimates for product, composition and commutators}
For the proofs of the following propositions, see 
\cites{bahouri-chemin-danchin,danchin2,danchin-mucha,runst-sickel}
and the references therein. 
\begin{prop}
\label{prop:prd}
Let $\nu\ge0$ and 
$\displaystyle -\min(\frac{n}{p},\frac{n}{p'})<\sigma\le\frac{n}{p}-\nu$.
The following product law holds:
\begin{equation*}
\|uv\|_{\dot B^{\sigma}_{p,1}}
\le C \|u\|_{\dot B^{\frac{n}{p}-\nu}_{p,1}}
      \|v\|_{\dot B^{\sigma+\nu}_{p,1}}.
\end{equation*}
\end{prop}
\begin{prop}
\label{prop:comp}
Let $F:I\to \R$ be a smooth function 
(with $I$ an open interval of $\R$ containing 0) 
vanishing at 0. Then for any $s>0$, $1\le p\le\infty$ 
and interval $J$ compactly supported 
in $I$ there exists a constant $C$ such that 
\begin{equation*}
\|F(a)\|_{\dot B^{s}_{p,1}}
\le C \|a\|_{\dot B^{s}_{p,1}}
\end{equation*}
for any $a\in\dot B^{s}_{p,1}$ with values in $J$. 
In addition, if $a_1$ and $a_2$ are two
such functions and $s=\Frac{n}{p}$ then we have
\begin{equation*}
\|F(a_2)-F(a_1)\|_{\dot B^{s}_{p,1}}
\le C \|a_2-a_1\|_{\dot B^{s}_{p,1}}.
\end{equation*}
\end{prop}
\begin{prop}
\label{prop:comm1}
Assume that $\sigma$, $\nu$ and $p$ are such that 
\begin{equation}
\label{cond:comm}
1\le p \le \infty, \quad
0\le\nu\le\frac{n}{p} \ \text{and} \ 
-\min(\frac{n}{p},\frac{n}{p'})-1 <\sigma\le\frac{n}{p}-\nu.
\end{equation}
Then there exists a constant $C$ depending only on $\sigma$, $\nu$, 
$p$ and $n$ such that for all 
$k\in\{1,...,n\}$, 
we have for some sequence $(c_j)_{j\in\Z}$ with
$\|c_j\|_{\ell^1}=1$
\begin{equation*}
\|\partial_k [a,\phi_j]w\|_{L^p}
\le C c_j 2^{-j\sigma} \|\nabla a\|_{\dot B^{\frac{n}{p}-\nu}_{p,1}}
      \|v\|_{\dot B^{\sigma+\nu}_{p,1}} 
      \quad \text{for all} \ j\in\Z.
\end{equation*}
\end{prop}
\begin{prop}
\label{prop:comm2}
Let $A(D)$ be a Fourier multiplier of degree 0. 
Then the following estimate holds.
\begin{equation*}
\|[A(D),q]w\|_{\dot B^{\sigma+1}_{p,1}}
\le C \| \nabla q\|_{\dot B^{\frac{n}{p}-\nu}_{p,1}}
      \|w\|_{\dot B^{\sigma+\nu}_{p,1}},
\end{equation*}
whenever 
\begin{equation}
1\le p \le \infty, \quad
\nu\geq 0 \quad \text{and} \ 
-\min\bigl(\frac{n}{p},\frac{n}{p'}\bigr)-1 <\sigma\le\frac{n}{p}-\nu.
\end{equation}
\end{prop}


\subsection{Estimates of flow}

We here recall  {\sl flow estimates} that have been proved in 
\cites{danchin2,danchin-mucha}. 
\begin{prop}\label{prop:flow1}
Let $1\le p<\infty$ and $v\in E_p(T)$.
There exists a positive constant $\tilde c$ (independent of $T$) such that if 
\begin{equation}\nonumber
\int_0^T \|Dv\|_{\dot B^{\frac{n}{p}}_{p,1}} dt\le \tilde c
\end{equation}
then for all $t\in[0,T]$, we have
\begin{equation}\nonumber
  \|\id-\adj(DX_v(t))\|_{\dot B^{\frac{n}{p}}_{p,1}} 
\lesssim \|D\overline v\|_{L^1_T(\dot B^{\frac{n}{p}}_{p,1})},
\end{equation}
\begin{equation}\nonumber
\|\id - A_v(t)\|_{\dot B^{\frac{n}{p}}_{p,1}} 
\lesssim \|D\overline v\|_{L^1_T(\dot B^{\frac{n}{p}}_{p,1})},
\end{equation}
\begin{equation}\nonumber
\|J_v^{\pm1}(t)-1\|_{\dot B^{\frac{n}{p}}_{p,1}} 
\lesssim \|D\overline v\|_{L^1_T(\dot B^{\frac{n}{p}}_{p,1})}.
\end{equation}
Furthermore, if $\overline w$ is
a vector field such that 
$\overline w \in L^1(0,T ; \dot B^{\frac{n}{p}}_{p,1})$,
then
\begin{equation}\nonumber
\|\adj(DX_v)D_{A_v}(\overline w) - D(\overline w)\|_{\dot B^{\frac{n}{p}}_{p,1}} 
\lesssim \|D\overline v\|_{L^1_T(\dot B^{\frac{n}{p}}_{p,1})}
 \|D\overline w\|_{L^1_T(\dot B^{\frac{n}{p}}_{p,1})},
\end{equation}
\begin{equation}\nonumber
\|\adj(DX_v)\div_{A_v}(\overline w) 
   - \div \overline w\, \id\|_{\dot B^{\frac{n}{p}}_{p,1}} 
\lesssim \|D\overline v\|_{L^1_T(\dot B^{\frac{n}{p}}_{p,1})}
         \|D\overline w\|_{L^1_T(\dot B^{\frac{n}{p}}_{p,1})}.
\end{equation}
\end{prop}

\begin{prop}\label{prop:flow2}
Let $1\le p<\infty$ and 
$\overline v_1$ and $\overline v_2\in E_p(T)$
satisfying 
\begin{equation}\nonumber
\int_0^T \|Dv\|_{\dot B^{\frac{n}{p}}_{p,1}} dt\le \tilde c
\end{equation}
and $\dv:=\overline v_2-\overline v_1$.
Then for all $t\in[0,T]$, we have
\begin{equation}\nonumber
\|A_{v_2}(t)-A_{v_1}(t)\|_{\dot B^{\frac{n}{p}}_{p,1}} 
\lesssim \|D\dv\|_{L^1_T(\dot B^{\frac{n}{p}}_{p,1})},
\end{equation}
\begin{equation}\nonumber
\|\adj(DX_{v_2}(t))-\adj(DX_{v_1}(t))\|_{\dot B^{\frac{n}{p}}_{p,1}} 
\lesssim \|D\dv\|_{L^1_T(\dot B^{\frac{n}{p}}_{p,1})},
\end{equation}
\begin{equation}\nonumber
\|J_{v_2}^{\pm1}(t)-J_{v_1}^{\pm1}(t)\|_{\dot B^{\frac{n}{p}}_{p,1}} 
\lesssim \|D\dv\|_{L^1_T(\dot B^{\frac{n}{p}}_{p,1})}.
\end{equation}
\end{prop}


\subsection{Lagrangean coordinates}
Let $X$ be a $\cC^1$-diffeomorphism over $\R^n$. 
For a vector-valued function $H:\R^n\to\R^m$, denote
$\bar H(y):=H(x)$ with $x=X(y)$.
The chain rule states that 
\begin{equation}
\label{chain}D_y\bar H(y) = D_xH(X(y)) \cdot D_yX(y)\quad\hbox{and}\quad
\nabla_y\bar H(y) = \nabla_yX(y) \cdot \nabla_xH(X(y)).
\end{equation}
Hence,  setting $A(y)=(D_yX(y))^{-1}=D_xX^{-1}(X(y)),$ we have 
\begin{equation}\nonumber
D_xH(X(y)) = D_y\bar H(y) \cdot A(y)\quad\hbox{and}\quad 
\nabla_x H(X(y))={}^t\! A(y)\cdot\nabla_y\bar H(y).
\end{equation}

\begin{prop}[\cite{danchin2}\cite{danchin-mucha}]
\label{flow-regr}
Let $X$ be a globally bi-Lipschitz
diffeomorphism of $\R^n$ and 
$(s,p)$ with $1\le p<\infty$ 
and $-\Frac{n}{p'}<s\le\Frac{n}{p}\cdotp$ Then $a \mapsto a\circ X$ is a 
self-map over $\dot B^s_{p,1}$ in  the following cases:
\begin{enumerate}
\item $s\in(0,1)$,
\item $s\in(-1,0]$ and $J_{X^{-1}}$ is in the 
multiplier space $\mathcal{M}(\dot B^s_{p,1})$,
\item $s\ge1$ and $(DX-\id)\in \dot B^s_{p,1}$.
\end{enumerate}
\end{prop}

\begin{prop}[\cite{danchin2}\cite{danchin-mucha}]
\label{prop:l-co}
Let $K$ be a $\cC^1$-scalar function over $\R^n$ and
$H$ be a $\cC^1$-vector field.
If $X$ is a $\cC^1$-diffeomorphism such that $J:=\det(D_yX)>0$,
then
\begin{equation}\nonumber
\overline{\nabla_x K} = J^{-1} \divy ( \adj (D_yX)\overline K),
\end{equation}
\begin{equation}\nonumber
\overline{\divx H} = J^{-1} \divy ( \adj (D_yX)\overline H),
\end{equation}
where $\adj(D_yX)$ is the adjugate matrix of $D_yX$.
\end{prop}

{}From the above proposition, we infer the following relations:
\begin{equation*}
\begin{split}
\overline{\Delta_x u} 
&= J^{-1} \divy (\adj(D_yX)\overline{\nabla_xu}) \\
&= J^{-1} \divy (\adj(D_yX) (^tA) \nabla_y\overline u), \\
\end{split}
\end{equation*}
\begin{equation*}
\begin{split}
\overline{\nabla_x\divx u} 
&= J^{-1} \divy (\adj(D_yX)\overline{\div_xu}) \\
&= J^{-1} \divy (\adj(D_yX) (^tA): \nabla_y\overline u), \\
\end{split}
\end{equation*}
\begin{equation*}
\overline{\nabla_x P} 
= J^{-1} \divy (\adj(D_yX)\overline{P}).
\end{equation*}

\begin{lem} Let $z:[0,T]\times\R^n\to\R^m$ and $X:[0,T]\times\R^n\to\R^n$ be 
differentiable functions with, in addition, $X(t):\R^n\to\R^n$ being a $\cC^1$ diffeomorphism
for all $t\in\R.$  Then the following relation holds:
$$
\partial_t(J\overline z) =J\,\overline{(\partial_tz+\divx(z u))}.
$$
\end{lem}
\begin{proof}
The proof is based on the following \emph{Jacobi formula}:
\begin{equation*}
\frac{d}{dt}\det A= \det A \: \tr \Bigl( A^{-1} \frac{dA}{dt}\Bigr)
\end{equation*}
that holds true whenever $A:[0,T]\to\cM_n(\R)$ is differentiable and $A(t)$ is invertible for all $t\in[0,T].$
\medbreak
Now, applying Jacobi  formula to $A(t)=D_yX(y,t),$ and using Leibniz rule,   we discover that 
\begin{equation*}
\begin{split}
\partial_t(J\overline z) 
&= (\partial_t J)\overline z + J\partial_t\overline z \\
&= J \tr \Bigl((D_yX)^{-1}\frac{dD_yX}{dt}\Bigr)\overline z + J\partial_t\overline z.
\end{split}
\end{equation*}
Since 
\begin{equation*}
\frac{dD_yX}{dt} = D_y \frac{dX}{dt} = D_y\overline u = \overline{D_xu}\cdot D_yX,
\end{equation*}
we thus have 
$$\partial_t(J\overline z) = J \,\overline z \,\overline{\divx u}
   + J(\overline{\partial_tz}+\overline{D_xz}\cdot\overline u), $$
   whence the desired equality.
\end{proof}

Applying the above lemma to $z=\rho,$ $z=\rho u$  or $z=E,$ we thus get 
\begin{equation*}
\begin{split}
&\overline{(\partial_t\rho+\divx(\rho u))} 
 = J^{-1}\partial_t(J\overline\rho), \\
&\overline{\partial_t(\rho u)+\divx(\rho u\otimes u)}
=J^{-1}\partial_t(J\overline\rho\overline u), \\
&\overline{\partial_t(E)+\divx(uE)}
=J^{-1}\partial_t(J\overline E).
\end{split}
\end{equation*}
{}From those three relations, it is now clear that if $(\rho,u,E)$ satisfies 
 \eqref{CNSE} then $(\overline\rho,\overline u,\overline E)$ fulfills \eqref{CNSE-l}.  
 
\begin{rem} 
Integrating against test functions, it is possible to considerably weaken the assumptions on $z.$
\end{rem}

\centerline{Ackowledgement}
The first author is supported by Research Fellowship for Young Scientists
of the Japan Society for the Promotion of Science (JSPS).


\begin{bibdiv}
 \begin{biblist}


\bib{bahouri-chemin-danchin}{book}{
   author={H., Bahouri\text{,}},
   author={J.-Y., Chemin\text{,}},
   author={R., Danchin\text{,}},
   title={Fourier analysis and nonlinear partial differential equations},
   series={Grundlehren der Mathematischen Wissenschaften},
   volume={343},
   publisher={Springer},
   place={Heidelberg},
   date={2011},
   pages={xvi+523},
}


\bib{chen-miao-zhang1}{article}{
   author={Q., Chen\text{,}},
   author={C, Miao\text{,}},
   author={Z, Zhang\text{,}},
   title={Well-posedness in critical spaces for the compressible
   Navier-Stokes equations with density dependent viscosities},
   journal={Rev. Mat. Iberoam.},
   volume={26},
   date={2010},
   number={3},
   pages={915--946},
}

\bib{chen-miao-zhang10}{article}{
   author={Q., Chen\text{,}},
   author={C, Miao\text{,}},
   author={Z, Zhang\text{,}},
   title={Global well-posedness for compressible Navier-Stokes equations
   with highly oscillating initial velocity},
   journal={Comm. Pure Appl. Math.},
   volume={63},
   date={2010},
   number={9},
   pages={1173--1224},
}


\bib{chen-miao-zhang2}{article}{
   author={Q., Chen\text{,}},
   author={C, Miao\text{,}},
   author={Z, Zhang\text{,}},
   title={On the ill-posedness of the compressible Navier-Stokes 
   equations in the critical Besov spaces},
   journal={arXiv:1109.6092},
}



\bib{charve-danchin}{article}{
   author={F., Charve\text{,}},
   author={R., Danchin\text{,}},
   title={A global existence result for the compressible Navier-Stokes
   equations in the critical $L\sp p$ framework},
   journal={Arch. Ration. Mech. Anal.},
   volume={198},
   date={2010},
   number={1},
   pages={233--271},
}

\bib{danchin00}{article}{
   author={R., Danchin\text{,}},
   title={Global existence in critical spaces for compressible Navier-Stokes
   equations},
   journal={Invent. Math.},
   volume={141},
   date={2000},
   number={3},
   pages={579--614},
}

\bib{danchin1}{article}{
   author={R., Danchin\text{,}},
   title={Local theory in critical spaces for compressible viscous and
   heat-conductive gases},
   journal={Comm. Partial Differential Equations},
   volume={26},
   date={2001},
   pages={1183--1233},
}


\bib{D2}{article}{
   author={R., Danchin\text{,}},
title={On the uniqueness in critical spaces for compressible Navier-Stokes equations},
journal={NoDEA Nonlinear Differential Equations Appl.}, 
volume={12},
number={1},
pages={111--128},
date={2005},
}

\bib{danchin07}{article}{
   author={R., Danchin\text{,}},
   title={Well-posedness in critical spaces for barotropic viscous fluids
   with truly not constant density},
   journal={Comm. Partial Differential Equations},
   volume={32},
   date={2007},
   pages={1373--1397},
}



\bib{danchin2}{article}{
   author={R., Danchin\text{,}},
   title={A Lagrangian approach for the compressible Navier-Stokes equations},
   journal={To appear in Annales de l'Institut Fourier},
}




\bib{danchin-mucha}{article}{
   author={R., Danchin\text{,}},
   author={P. B., Mucha\text{,}},
   title={A Lagrangian approach for the incompressible Navier-Stokes
   equations with variable density},
   journal={Comm. Pure Appl. Math.},
  volume={65},
   date={2012},
   number={10},
   pages={1458--1480},
}


\bib{fujita-kato}{article}{
   author={H., Fujita\text{,}},
   author={T., Kato\text{,}},
   title={On the Navier-Stokes initial value problem. I},
   journal={Arch. Rational Mech. Anal.},
   volume={16},
   date={1964},
   pages={269--315},
}


\bib{haspot}{article}{
   author={B., Haspot\text{,}},
   title={Well-posedness in critical spaces for the system of compressible
   Navier-Stokes in larger spaces},
   journal={J. Differential Equations},
   volume={251},
   date={2011},
   pages={2262--2295},
}





\bib{matsumura-nishida}{article}{
   author={A., Matsumura\text{,}},
   author={T., Nishida\text{,}},
   title={The initial value problem for the equations of motion of viscous
   and heat-conductive gases},
   journal={J. Math. Kyoto Univ.},
   volume={20},
   date={1980},
   pages={67--104},
}

\bib{mucha}{article}{
   author={P. B., Mucha\text{,}},
   title={The Cauchy problem for the compressible Navier-Stokes equations in
   the $L\sb p$-framework},
   journal={Nonlinear Anal.},
   volume={52},
   date={2003},
   number={4},
   pages={1379--1392},
}

\bib{nash}{article}{
   author={J., Nash\text{,}},
   title={Le probl\`eme de Cauchy pour les \'equations diff\'erentielles
   d'un fluide g\'en\'eral},
   journal={Bull. Soc. Math. France},
   volume={90},
   date={1962},
   pages={487--497},
}

\bib{runst-sickel}{book}{
   author={T., Runst\text{,}},
   author={W., Sickel\text{,}},
   title={Sobolev spaces of fractional order, Nemytskij operators, and
   nonlinear partial differential equations},
   series={de Gruyter Series in Nonlinear Analysis and Applications},
   volume={3},
   publisher={Walter de Gruyter \& Co.},
   place={Berlin},
   date={1996},
   pages={x+547},
}

\bib{triebel}{book}{
   author={H., Triebel\text{,}},
   title={Theory of function spaces},
   series={Monographs in Mathematics},
   volume={78},
   publisher={Birk\-h\"auser Verlag},
   place={Basel},
   date={1983},
   pages={284},
   isbn={3-7643-1381-1},
}

\bib{valli}{article}{
   author={A, Valli\text{,}},
   title={An existence theorem for compressible viscous fluids},
   journal={Ann. Mat. Pura Appl. (4)},
   volume={130},
   date={1982},
   pages={197--213},
}

\bib{valli-zajaczkowski}{article}{
   author={A., Valli\text{,}},
   author={W. M., Zaj{\polhk{a}}czkowski\text{,}},
   title={Navier-Stokes equations for compressible fluids: global existence
   and qualitative properties of the solutions in the general case},
   journal={Comm. Math. Phys.},
   volume={103},
   date={1986},
   number={2},
   pages={259--296},
}


  \end{biblist}
\end{bibdiv}
\end{document}